%% file: Gassiat_Gess_Lions_Souganidis-2019-Speed_of_propagation.tex
\documentclass{amsart}
\usepackage[latin9]{inputenc}
\usepackage{geometry}
\usepackage{verbatim}
\usepackage{amstext}
\usepackage{amsthm}
\usepackage{amssymb}
\usepackage{graphicx}
\usepackage{mathabx} 
\makeatletter 
\numberwithin{equation}{section}
\numberwithin{figure}{section}
\theoremstyle{plain}
\newtheorem{thm}{\protect\theoremname}
  \theoremstyle{definition}
  \newtheorem{defn}[thm]{\protect\definitionname}
  \theoremstyle{plain}
  \newtheorem{lem}[thm]{\protect\lemmaname}
  \theoremstyle{plain}
  \newtheorem{cor}[thm]{\protect\corollaryname}
  \theoremstyle{plain}
  \newtheorem{prop}[thm]{\protect\propositionname}
  \theoremstyle{remark}
  
\numberwithin{thm}{section}

\usepackage{amsfonts}
\usepackage{color}

\newcommand{\blue}[1]{#1}

\newcommand{\R}{ {\mathbb{R}} }
\newcommand{\E}{ {\mathbb{E}} }
\renewcommand{\P}{ {\mathbb{P}} }

\theoremstyle{plain}

\renewcommand{\d}{\delta}

\renewcommand{\and}{\quad\textrm{ and }\quad}
\renewcommand{\S}{\mathcal{S}}
\renewcommand{\P}{\mathbb{P}}

\newcommand{\N}{\mathbb{N}}

\newcommand{\mcA}{\mathcal{A}}
\newcommand{\sgn}{\textrm{sgn}}

\numberwithin{equation}{section}

  \providecommand{\corollaryname}{Corollary}
  \providecommand{\definitionname}{Definition}
  \providecommand{\lemmaname}{Lemma}
  \providecommand{\propositionname}{Proposition}
  \providecommand{\remarkname}{Remark}
\providecommand{\theoremname}{Theorem}

\makeatother

  \providecommand{\corollaryname}{Corollary}
  \providecommand{\definitionname}{Definition}
  \providecommand{\lemmaname}{Lemma}
  \providecommand{\propositionname}{Proposition}
  \providecommand{\remarkname}{Remark}
\providecommand{\theoremname}{Theorem}

\date{\today}
\begin{document}

\title[speed of propagation for stochastic Hamilton-Jacobi equations]{speed of propagation for Hamilton-Jacobi equations with multiplicative rough time dependence and convex Hamiltonians}

\author[P. Gassiat]{Paul Gassiat}
\address{Ceremade, Universit\'e de Paris-Dauphine\\
Place du Mar\'echal-de-Lattre-de-Tassigny\\
75775 Paris cedex 16, France}
\email{gassiat@ceremade.dauphine.fr}

\author[B. Gess]{Benjamin Gess}
\address{Max Planck Institute for Mathematics in the Sciences, Inselstrasse 22, 04103 Leipzig and Faculty of Mathematics,
University of Bielefeld,
33615 Bielefeld,
Germany}
\email{benjamin.gess@gmail.com}

\author[P-L. Lions]{Pierre-Louis Lions}
\address{Coll\`{e}ge de France,
11 Place Marcelin Berthelot, 75005 Paris, 
and  
CEREMADE, 
Universit\'e de Paris-Dauphine,
Place du Mar\'echal de Lattre de Tassigny,
75016 Paris, France}
\email{lions@ceremade.dauphine.fr}

\author[P. E. Souganidis]{Panagiotis E. Souganidis}
\address{Department of Mathematics 
University of Chicago, 
5734 S. University Ave.,
Chicago, IL 60637, USA}
\email{souganidis@math.uchicago.edu}


\begin{abstract}
We show that the initial value problem for Hamilton-Jacobi equations with multiplicative rough time dependence, typically stochastic,  and convex Hamiltonians satisfies finite speed of propagation. We prove that in general the range of dependence is bounded by a multiple of the length of the ``skeleton'' of the path, that is a piecewise linear  path obtained by connecting  the successive extrema of the original one. When the driving path is a Brownian motion, we prove that its skeleton has almost surely finite length. We also discuss the optimality of the estimate. 
\end{abstract}

\maketitle

\section{Introduction}

We consider the initial value problem  for Hamilton-Jacobi equations with multiplicative rough time dependence, that is 
\begin{equation}
du=H(Du,x)\cdot d{\xi}\ \ \text{in }\ \R^{d}\times (0,T] \qquad u(\cdot,0)=u_0 \ \  \text{in} \ \  \R^d,\label{eq:HJ-intro}
\end{equation}
with 
\begin{equation}\label{takis1}
\text{$H:\R^d \times \R^d \to \R$ convex and Lipschitz continuous in the first argument}
\end{equation}
and
\begin{equation}\label{takis2}
\text{$\xi \in C_0([0,T]),$}
\end{equation}
where $C_0([0,T])$ denotes the space of continuous paths $\xi: [0,T] \to \R$ such that $\xi(0)=0.$  
\smallskip

When $\xi$ is a $C^1$ or $\text{BV}$-path,   \eqref{eq:HJ-intro} is the standard Hamilton-Jacobi equation that is studied using the Crandall-Lions theory of viscosity solutions. For such paths, in place of \eqref{eq:HJ-intro}  we will often write 
\begin{equation}\label{takis10}
u_t=H(Du,x)\dot {\xi}\ \ \text{in }\ \R^{d}\times (0,T] \qquad u(\cdot,0)=u_0 \ \  \text{in} \ \  \R^d\end{equation}

When $\xi$ is merely continuous, in \eqref{eq:HJ-intro} $\cdot$ simply denotes the way the path enters the equation. When $\xi$ is  a Brownian motion, then  $\cdot d\xi$ stands for the classical Stratonovich differential.
\smallskip

Lions and Souganidis introduced in \cite{LS98a} the notion of stochastic or pathwise viscosity solutions for a general class of equations which contain \eqref{eq:HJ-intro}
as a special case and studied its well-posedness; for this as well as further properties see Lions-Souganidis \cite{LS98a,LS98b,LS-college, LS18+, S16}.
\smallskip

One of the questions raised in \cite{S16} was whether  \eqref{eq:HJ-intro} has a finite speed of propagation, which is one of the important characteristics  of the hyperbolic nature of the equations for regular paths. Roughly speaking, finite speed of propagation means that, if two solutions agree at some time  in a  ball, then they agree on a forward cone with a time dependent radius. 
\smallskip

\blue{Hamilton-Jacobi equations with rough time dependence appear for example in the theory of mean field games with common noise (e.g.\ Lasry, Lions \cite{LL061,LL061}, Carmona, Delarue \cite{CD18}), as limits of interacting particle systems with common noise (e.g.\ \cite{LPS13,GS15,GS17,FG19,CG19}), in pathwise optimal control, and in stochastic geometric motions (e.g.\ Lions, Souganidis \cite{LS98a,LS98b,LS-college, LS18+, S16}). In particular, we recall that the stochastic mean curvature flow
  $$dv-|\mathrm{D}v|\mathrm{div}\left(\frac{\mathrm{D}v}{|\mathrm{D}v|}\right)dt+ \phi(x)|\mathrm{D}v|\circ d\beta_t =0\quad\text{in }\R^d \times(0,T],$$
describes the geometric motion of a surface by mean curvature with a spatially homogeneous Brownian perturbation acting in the normal direction to the surface (see Souganidis \cite{S16} and Es-Sarhir, Renesse \cite{ER12} for details and notation). In this case of geometric stochastic PDE, the problem of finite speed of propagation is linked to the speed of stochastic front propagation.}
\smallskip

A partial result in this direction was shown in Lions and Souganidis~ \cite{LS98b} (see also Souganidis \cite{S16}), while Gassiat showed in \cite{Gas17} that, in general,  when $H$ is neither convex nor concave \eqref{eq:HJ-intro} does not have the finite speed of propagation property.

\smallskip 

In this work, assuming \eqref{takis1} and \eqref{takis2}, we establish finite speed of propagation in the sense formulated precisely next.
\vskip.05in

Given $T>0$ and $H :\R^d \times \R^d \to \R$ let 
\begin{align}
\rho_{H}(\xi,T):=\sup\Big\{ R\ge0:\  \text{there exist solutions} & \ \ u^{1},u^{2}\mbox{ of }\eqref{eq:HJ-intro} \text{ and } x \in \R^d, \\  \text{ such that  }  u^{1}(\cdot, 0)=u^{2}(\cdot, 0) &\mbox{ in }B_R(x)\ \label{eq:defR-2}
  \text{and }u^{1}(x, T)\ne u^{2}(x, T)\Big\}, \nonumber 
\end{align}
where $B_R(x)$ is the ball in $\R^d$ centered at $x$ with radius $R$.
\smallskip

The  classical theory  for Hamilton-Jacobi equations (see Lions~\cite{LionsBook} and Crandall and Lions~\cite{CL92}) yields that, if $\xi$ is a $C^1$- or, more generally, a  $\text{BV}$-path,  then 
\begin{equation}
\rho_{H}(\xi,T)\leq L \|\xi\|_{TV([0,T])},\label{eq:TrivialBound}
\end{equation}
where 
\[
\|\xi\|_{TV([0,T])}:=\sup_{0=t_{0}\leq\ldots\leq t_{n}=T}\sum_{i=0}^{n-1}\left|\xi(t_{i+1})-\xi(t_{i})\right| 
\]
is the total variation semi-norm of $\xi$ and $L$ is the Lipschitz constant of $H$. It is easy to see that \eqref{eq:TrivialBound} is sharp when $\dot {\xi}\equiv1$.
\smallskip


For general rough, that is only continuous,  signal $\xi$ it was shown  in \cite{LS98b},  \cite{S16} that, if $H(p,x)=H_{1}(p)-H_{2}(p)$, where $H_{1}$, $H_{2}$ satisfy \eqref{takis1} with Lipschitz constant $L$ and $H_{1}(0)=H_{2}(0)=0$, then, for any constant $A$, if
\[
u(0,\cdot)\equiv A\text{ on }B_{R}(0), 
\]
then 
\begin{equation}\label{eq:fsp-2}
  u(t,\cdot)\equiv A\text{ on }B_{R(t)}(0),\quad\text{for }R(t):=R-L(\max_{s\in[0,\blue{t}]}\xi(s)-\min_{s\in[0,\blue{t}]}\xi(s)). 
\end{equation}
This does not, however, imply a finite range of dependence 
\blue{in the sense of a bound on $\rho_{H}(\xi,T)$}. 
\smallskip

In fact, it was shown in \cite{Gas17} that when   $H(p)=|p_{1}|-|p_{2}|$ equality is attained in \eqref{eq:TrivialBound} for all continuous $\xi$, a fact which implies that there is no finite domain of dependence if $\xi\not\in BV([0,T])$. In other words, the counter-example in \cite{Gas17} shows that for non-convex Hamiltonian $H$, all of the oscillations of $\xi$, measured in terms of the $TV$-norm, are relevant for the dynamics of \eqref{eq:HJ-intro}.
\smallskip

In contrast, in this paper we show that, if $H$ is convex, there is an estimate, which is  better than \eqref{eq:TrivialBound}, and, in particular, implies that the \blue{range} of dependence $\rho_{H}(\xi,T)$ is almost surely finite when $\xi$ is a Brownian path. This new  bound 
relies on a better understanding of which oscillations of the signal $\xi$ are effectively relevant for the dynamics of \eqref{eq:HJ-intro}. 
\smallskip

In this spirit, we prove that, if $H$ is convex, then $\xi$ can be replaced by its skeleton. This is a reduced path $R_{0,T}(\xi)$ 
which keeps track solely of the oscillations of $\xi$ that are relevant for the dynamics of \eqref{eq:HJ-intro} without changing the solution to \eqref{eq:HJ-intro}. Hence, in the convex case only the oscillations of $\xi$ encoded in $R_{0,T}(\xi)$ are relevant for \eqref{eq:HJ-intro}. \blue{The proof of this fact presented in this work relies on the representation of viscosity solutions to Hamilton-Jacobi equations via the associated optimal control problem. At this point, the convexity of the Hamiltonian $H$ is essential. In contrast, in non-convex cases a representation in terms of a differential game would have to be used instead, which was the basis of the counter-example in \cite{Gas17}.} In the one-dimensional setting and for smooth, strictly convex, $x$-independent Hamiltonians a related result has been obtained independently and by different methods in  Hoel, Karlsen, Risebro, and Storr{\o}sten~\cite{HKRS}. 
\smallskip

We also establish that the reduced path 
of a Brownian motion
has almost surely finite variation, a fact which implies that $\rho_{H}(\xi,T)$ is almost surely finite.
\smallskip

Given $\xi \in C_0([0,T])$, \blue{$\arg\min_{[a,b]}$ ($\arg\max_{[a,b]}$ resp.) denotes the set of minimal (maximal resp.) points of $\xi$ on the interval $[a,b]\subseteq[0,T]$} and the sequence $(\tau_{i})_{i\in\mathbb{Z}}$ of successive extrema of $\xi$  is defined by 
\begin{equation}
\tau_{0}:=\sup\left\{ t\in[0,T],\;\;\xi(t)=\max_{0\leq s\leq T}\xi(s)\mbox{ or }\xi(t)=\min_{0\leq s\leq T}\xi(s)\right\} ,\label{eq:extrema1}
\end{equation}
and, for all  $ i\geq0$, 
\begin{equation}
\tau_{i+1}=\left\{ \begin{array}{ll}
\blue{\sup}\arg\max_{[\tau_{i},T]}\xi & \mbox{ if }\ \ \xi(\tau_{i})<0,\\[1.5mm]
\blue{\sup}\arg\min_{[\tau_{i},T]}\xi & \mbox{ if } \ \ \xi(\tau_{i})>0,
\end{array}\right.\label{eq:extrema2}
\end{equation}
and, for all  $ i\leq0$,   
\begin{equation}
\tau_{i-1}=\left\{ \begin{array}{ll}
\blue{\inf}\arg\max_{[0,\tau_{i}]}\xi & \mbox{ if } \ \  \xi(\tau_{i})<0,\\[1.5mm]
\blue{\inf}\arg\min_{[0,\tau_{i}]}\xi & \mbox{ if }\ \  \xi(\tau_{i})>0.
\end{array}\right.\label{eq:extrema3}
\end{equation}
\smallskip 
The skeleton (resp. full skeleton) or  reduced (resp. fully reduced) path ${R}_{0,T}(\xi)$ (resp. $\tilde{R}_{0,T}(\xi)$) of $\xi \in C_0([0,T])$ is defined as follows (see Figure \ref{fig:reduced_path}). 
\begin{defn}
Let $\xi \in C_0([0,T])$. 

(i)~The reduced path $R_{0,T}(\xi)$ is a piecewise linear function which agrees  with $\xi$ on $(\tau_{i})_{ i\in\mathbb{Z}}$.

(ii)~The fully reduced path  $\tilde{R}_{0,T}(\xi)$ is a piecewise linear function  agreeing with 
 $\xi$ on $(\tau_{-i})_{i\in\mathbb{N}} \cup \{T\}$.
 
(iii)~ A path $\xi \in C_0([0,T])$ is reduced (resp. fully reduced) if $\xi=R_{0,T}(\xi)$ (resp. $\xi=\tilde{R}_{0,T}(\xi)$). 
\end{defn}
\begin{figure}
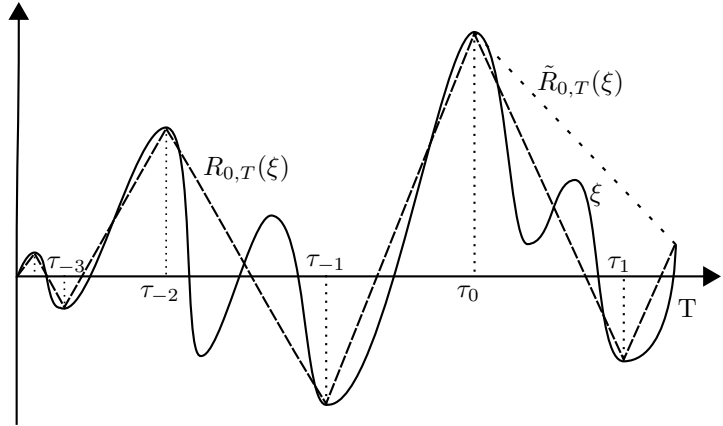
\caption{The (fully) reduced path}
\label{fig:reduced_path}
\end{figure}
\blue{We emphasize that the reduced and the fully reduced paths coincide prior to the global extremum $\tau_0$. While the reduced path captures the max-min fluctuations also after $\tau_0$, the fully reduced path is affine linear on $[\tau_0,T]$ and, in this sense, is more "reduced" .}

Let $u^{\xi}$ be the solution to \eqref{eq:HJ-intro}.  We show in Theorem \ref{thm:reduction} in the next section  that 
\begin{equation}
u^{\xi}(\cdot, T)=u^{R_{0,T}(\xi)}(\cdot, T),\label{eq:intro-reduced}
\end{equation}
which immediately implies the following result. 
\begin{thm}\label{takis3}
Assume \eqref{takis1}. 
Then, 
for all $\xi \in C_0([0,T])$,  
\begin{equation} \label{eq:main-speed}
\rho_{H}(\xi,T)\leq L \ \|R_{0,T}(\xi)\|_{TV([0,T])}.
\end{equation}
\end{thm}
The second main result of the paper, which  is a probabilistic one and of independent interest, concerns the total variation of the reduced path of a Brownian motion. To state it, we introduce the random variable $\theta: [0,\infty) \to [0,\infty)$ given by 
\begin{equation}\label{takis10005}
\theta(a):=\inf \{ t\geq0:\;\;\max_{[0,t]}B-\min_{[0,t]}B=a\}.
\end{equation}
We prove  that the length of the reduced path is a random variable with almost Gaussian tails. We also show that if,  instead of fixing the time horizon $T$,  we fix the range, that is the maximum minus the minimum of $B$, then the length has Poissonian tails

\begin{thm}\label{takis4}
\label{thm:BM} Let $B$ be a Brownian motion and fix $T>0$. Then, for each $\gamma \in (0,2)$, there exists $C=C(\gamma, T)>0$ such that, for any $x\geq2$, 
\begin{equation}
\P\left(\left\Vert R_{0,T}(B)\right\Vert _{TV([0,T])}\geq x\right)\leq C \exp\left(-Cx^{\gamma}\right),\label{eq:RGaussian}
\end{equation}
and
\begin{equation}\label{eq:RPoisson}
\displaystyle{\underset{x\to\infty} \lim \frac{\ln \P\left(\left\Vert R_{0,\theta(1)}(B)\right\Vert _{TV([0,\theta(1)])}\geq x\right)}{x\ln(x) }= -1.} 
\end{equation}
\end{thm}
\smallskip
\blue{A related result, proving that the expectation of the total variation of the so-called piecewise linear oscillating running max/min function of Brownian motion is finite, has been obtained independently by Hoel, Karlsen, Risebro, and Storr{\o}sten in \cite{HKRS2}.}

We also study the sharpness of the upper bound. For simplicity we only treat the case $H(p)=|p|$. 
\begin{thm}\label{taki4}
\label{thm:opt} Let $H(p)=|p|$ on $\R^{d}$  with $d\ge1$. 
Then, for all $T>0$ and  $\xi \in C_0([0,T])$, 
\begin{equation}
\rho_{H}(\xi, T)\geq\|\tilde{R}_{0,T}(\xi)\|_{TV([0,T])}.
\end{equation}
When $d=1$, then  
\[
\rho_{H}(\xi, T)=\|\tilde{R}_{0,T}(\xi)\|_{TV([0,T])}.
\]
\end{thm}
The paper is organized as follows. In section~\ref{sec:Representation-with-reduced} we  improve upon results of \cite{LS98a,LS98b,LS-book} about representation formulae,  the control of the oscillations in time and the domain of dependence of the solutions of \eqref{takis10} with piecewise linear paths. We then extend these estimates by density to general continuous paths. In order to avoid stating many assumptions on $H$, we introduce a new condition about solutions of \eqref{takis10} which is  satisfied by the general class of Hamiltonians for which there is a well-posed theory of pathwise solutions as developed in \cite{LS18+}. All these lead to the proof of
Theorem~\ref{takis3}. In section 3 we discuss  the example which shows that the upper bound obtained in Theorem~\ref{takis3} is sharp. Section 4 is devoted to the study of ``random''  properties of the reduced path of the Brownian motion (Theorem~\ref{takis4}). 

\section{Reduction to the skeleton path and domain of dependence}
\label{sec:Representation-with-reduced}

\subsection*{Notation and preliminaries} For all $\xi\in C_0([0,T])$ and $u_{0}\in BUC(\R^{d})$, let $S^{\xi}_{\blue{H}}$ be the flow of solutions of \eqref{eq:HJ-intro}\blue{, that is, $S^{\xi}_H(s,t)u_0$ is the viscosity solution to \eqref{eq:HJ-intro} with initial condition $u_0$ at time $s$, i.e.\ $S^{\xi}_H(s,s)u_0=u_0$, and solving \eqref{eq:HJ-intro} driven by the signal $\xi$, in the viscosity sense, on $[s,\infty)$}. 
A simple rescaling shows that without loss of generality we may assume that $$L=1.$$ 
In view of \eqref{takis1} and the normalization of the Lipschitz constant we have 
\[
H(p,x)=\sup_{v\in B_1(0)} \left\{ p\cdot v- L(v,x) \right\},
\]
where $L(v,x)=\sup_{p \in \R^d} \left\{ p\cdot v- H(p,x) \right\}$.
\smallskip

We assume that $H$ satisfies all assumptions needed (see~\cite{CL92}) for 
$u_t= H(Du,x)\dot{\xi}$ to be well posed when $\xi$ is smooth and we denote  by  $S_{\pm H}(t):\text{BUC}(\R^d) \to \text{BUC}(\R^d)$ the solution operator when $\dot \xi\equiv \pm1$, that is, for $u_0\in  \text{BUC}(\R^d)$,
$S_{\pm H}(t)u_0$ is the unique solution of
\begin{equation}\label{takis100}
u_t=\pm H(Du, x) \ \text{in} \ \R^d \times (0,T] \quad u(\cdot,0)=u_0 \ \text{in} \ \R^d.\end{equation}
Moreover, for $t \leq 0$,   $S_H(t) := S_{-H}(-t)$. \blue{Hence, $S_{\pm H}(t)u_0 = S^{\xi_\pm}_H(0,t)u_0$, where $\xi_\pm(t):=\pm t$.}
\smallskip

Given $S,S' : \text{BUC}(\R^d)\to \text{BUC}(\R^d) $,  we say that $S \leq S'$ if  $Su \leq S'u$ for all $u$ in $\text{BUC}(\R^d)$.

\smallskip

In the sequel we write  $\xi_{s,t}:=\xi_{t}-\xi_{s}$ for  the increments of $\xi$ over the interval $[s,t]$. 
\smallskip

Let $\xi \in C([0,T])$ be a piecewise linear path, that is, for  a  partition $0=t_0 \leq \ldots\leq t_N =T$ of $[0,T]$, and $a_{i},b_{i}\in\R$, $i=1,\dots N$,  
$$\xi(t)=\sum_{i=0}^{N-1}1_{[t_{i},t_{i+1})}(a_{i}(t-t_{i})+b_{i}).$$

We then set
\begin{align*}
S_{H}^{\xi}(0,T):= & S_{H}^{\xi}(t_{N-1},t_{N})\circ\dots\circ S_{H}^{\xi}(t_{0},t_{1})
\end{align*}
and note that 
$$S_{H}^{\xi}(0,T)=S_{H}(\xi_{t_{N-1},t_{N}})\circ\dots\circ S_{H}(\xi_{t_{0},t_{1}}).$$

%
%
%

We show later that $\xi\mapsto S_{H}^{\xi}(0,\cdot)$ is uniformly continuous in sup-norm, which allows to extend $S_{H}^{\xi}(0,T)$ to all continuous $\xi$. 

\subsection*{Monotonicity properties}

The control representation of the solution $u$ of \eqref{takis100} (see, for example, Lions~\cite{LionsBook}) with $\xi_{t}\equiv t$ and $u_0 \in \text{BUC}(\R^d)$ is  
\[
u(x,t)=S_{H}(t)u_{0}(x)=\sup_{q\in\mcA}\left\{ u_{0}(X(t))-\int_{0}^{t}L(q(s), X(s))ds:\,X(0)=x,\,\dot{X}(s)=q(s) \ \text{for $s\in [0,t]$} \right\} ,
\]
and 
\[
S_{-H}(t)u_{0}(y)=\inf_{r\in\mcA}\left\{ u_{0}(Y(t))+\int_{0}^{t}L(r(s),Y(s))ds:\,Y(0)=y,\,\dot{Y}(s)=-r(s) \ \text{for $s\in [0,t]$}\right\},
\]
where $\mcA=L^{\infty}(\R_{+}; \overline B_1(0))$ is the set of controls. 
\smallskip

The next property is a refinement of an observation in \cite{LS-book}, \blue{see also Barron, Cannarsa, Jensen, Sinestrari \cite{BCJ99}}. 
\begin{lem}
\label{lem:reduction1} Fix  $t>0$ and  $u_{0}\in BUC(\R^{d})$. Then
\[
S_{H}(t)\circ S_{H}(-t)u_{0}\le u_{0}\le S_{H}(-t)\circ S_{H}(t)u_{0}.
\]
\end{lem}

\begin{proof}
Since the arguments are identical we only show the proof of the inequality on the left. 
\smallskip

We have
\begin{align*}
S_{H}(t)\circ S_{H}(-t)u_{0}(x)=\sup_{q\in\mcA}\inf_{r\in\mcA}\Big\{ & u_{0}(Y(t))+\int_{0}^{t}L(r(s),Y(s))ds-\int_{0}^{t}L(q(s),X(s))ds:\\
 & Y(0)=X(t),\,\dot{Y}(s)=-r(s),\,X(0)=x,\,\dot{X}(s)=q(s)\ \text{for $s\in [0,t]$}\Big\}.
\end{align*}
Given $q\in\mcA$  choose $r(s)=q(t-s)$ in the infimum above. Since  $Y(s)=X(t-s)$, it follows that 
\begin{align*}
S_{H}(t)\circ S_{H}(-t)u_{0}(x) & \le\sup_{q\in\mcA}\Big\{ u_{0}(X(0))+\int_{0}^{t}L(q(t-s),X(t-s))ds-\int_{0}^{t}L(q(s),X(s))ds:\\
 & \quad\quad\quad X(0)=x,\,\dot{X}(s)=q(s)  \ \text{for $s\in [0,t]$} \Big\}\\
 & =u_{0}(x).
\end{align*}
\end{proof}
The next result is an easy consequence of Lemma~\ref{lem:reduction1} and the definition of $S_{H}^{\xi}$ for piecewise linear paths.
\begin{lem}
\label{lem:bump}Let $\xi_{t}=1_{t\in[0,t_{1}]}(a_{0}t)+1_{t\in[t_{1},T]}(a_{1}(t-t_{1})+a_{0}t_{1})$.  If  $a_{0}\ge0$ and $a_{1}\le0$ (resp. $a_{0}\le0$ and  $a_{1}\ge0$), then 
\[
S_{H}^{\xi}(0,T)\ge S_{H}(\xi_{0,T}) \qquad (resp. \ \ S_{H}^{\xi}(0,T)\le S_{H}(\xi_{0,T}).)
\]
\end{lem}

\begin{proof}
Since the claim is immediate if $a_{0}=0$ or $a_{1}=0$, we assume next that 
$a_{0}>0$ and $a_{1}<0$ (see Figure \ref{fig:reduction1}). 

\begin{figure}
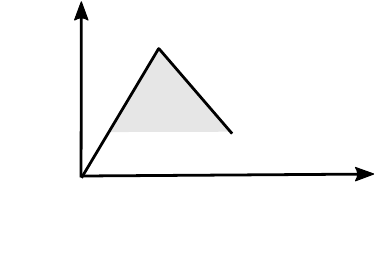
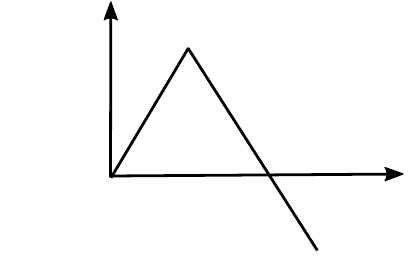
\caption{Reduction}
\label{fig:reduction1}
\end{figure}
\smallskip

If $\xi_{0,T}\le0$, then
\begin{align*}
S_{H}(a_{1}(T-t_{1})) & =S_{-H}(-a_{1}(T-t_{1}))=S_{-H}(-a_{1}(T-t_{1})-a_{0}t_{1})\circ S_{-H}(a_{0}t_{1})\\
 & =S_{-H}(-\xi_{0,T})\circ S_{-H}(a_{0}t_{1})=S_{H}(\xi_{0,T})\circ S_{H}(-a_{0}t_{1}),
\end{align*}
and, hence, in view of  Lemma \ref{lem:reduction1},
\begin{align*}
S_{H}^{\xi}(0,T) & =S_{H}(\xi_{0,T})\circ S_{H}(-a_{0}t_{1})\circ S_{H}(a_{0}t_{1}) 
  \ge S_{H}(\xi_{0,T}).
\end{align*}
If $\xi_{0,T}\ge0$ (see Figure \ref{fig:reduction1}), 
then, again using Lemma \ref{lem:reduction1}, we find 
\begin{align*}
S_{H}^{\xi}(0,T) & =S_{H}(a_{1}(T-t_{1}))\circ S_{H}(-a_{1}(T-t_{1})+a_{0}t_{1}+a_{1}(T-t_{1}))\\
 & =S_{H}(a_{1}(T-t_{1}))\circ S_{H}(-a_{1}(T-t_{1}))\circ S_{H}(a_{0}t_{1}+a_{1}(T-t_{1}))
  \le S_{H}(\xi_{0,T}).
\end{align*}
For the second inequality we note that $S_{-H}^{-\xi}(0,T)=S_{H}^{\xi}(0,T)$, $S_{-H}(-t)=S_{H}(t)$. It then follows from the  the first part that 
\[
S_{H}^{\xi}(0,T)=S_{-H}^{-\xi}(0,T)\ge S_{-H}(-\xi_{0,T})=S_{H}(\xi_{0,T}).
\]
\end{proof}
The next observation provides the first indication of the possible reduction encountered when using the max or min of a given path. For the statement, given piecewise linear path $\xi$, we set 
\[ \tau_{max}:=\sup\{t\in[0,T]:\ \xi_{t}=\max_{s\in[0,T]}\xi_{s}\}  \ \text{ and} \  \tau_{min}:=\blue{\inf}\{t\in[0,T]:\ \xi_{t}=\min_{s\in[0,T]}\xi_{s}\}.\]

\begin{lem}
\label{lem:reduction3} Fix  a piecewise linear path $\xi$.
 Then 
\[
S_{H}^{\xi}(\tau_{max},T)\circ S_{H}(\xi_{0,\tau_{max}})\leq S_{H}^{\xi}(0,T)\leq S_{H}(\xi_{\tau_{min},T})\circ S_{H}^{\xi}(0,\tau_{min}).
\]
\end{lem}

\begin{proof}
Since the proofs of both inequalities are similar, we only show the details for the first. 
\smallskip

Note that without loss of generality we may assume that $\sgn(\xi_{t_{i-1},t_{i}})=-\sgn(\xi_{t_{i},t_{i+1}})$ for all $[t_{i-1},t_{i+1}]\subseteq[0,\tau_{max}]$. 
\smallskip

It follows that, if  $\xi_{|[0,\tau_{max}]}$ is linear, then $S_{H}^{\xi}(0,\tau_{max})=S_{H}(\xi_{0,\tau_{max}})$.
\smallskip

If not, since $\xi_{0,\tau_{max}}\ge0$, there is an index $j$ such that $\xi_{t_{j-1},t_{j+1}}\ge0$ and $\xi_{t_{j-1},t_{j}}\le0$. It then follows from Lemma \ref{lem:bump} that 
\[
S_{H}^{\xi}(0,\tau_{max})\le S_{H}^{\tilde{\xi}}(0,\tau_{max}), 
\]
 where $\tilde{\xi}$ is piecewise linear and coincides with $\xi$ for all $t\in\{t_{i}:\ i\ne j\}$. 
\smallskip
 
A simple iteration yields $S_{H}^{\xi}(0,\tau_{max})\le S_{H}(\xi_{0,\tau_{max}})$, and, since $S_{H}^{\xi}(0,T)=S_{H}^{\xi}(\tau_{max},T)\circ S_{H}^{\xi}(0,\tau_{max})$, this concludes the proof.

\end{proof}
We combine the conclusions of the previous lemmata to establish the following monotonicity result. 
\begin{cor} \label{cor:Monotone}
Let $\xi,\zeta$ be piecewise linear, $\xi(0)=\zeta(0)$, $\xi(T)=\zeta(T)$ and $\xi\leq\zeta$ on $[0,T]$. Then
\begin{equation} \label{eq:Monotone}
S_{H}^{\xi}(0,T)\leq S_{H}^{\zeta}(0,T).
\end{equation}
\end{cor}

\begin{proof}
We assume that $\xi$ and $\zeta$ are piecewise linear on each interval $[t_i,t_{i+1}]$ of a joint subdivision $0=t_0 \leq \ldots \leq t_N=T$ of $[0,T]$.
\smallskip

If  $N=2$, we show that, for all  $\gamma \geq 0$ and all $a, b \in \R$,
\begin{equation} \label{eq:Monotone2}
S_H(a+\gamma) \circ S_H(b-\gamma) \leq S_H(a) \circ S_H(b).
\end{equation}
If $a\geq 0$, this follows from the fact that, in view of  Lemma \ref{lem:bump}, \[S_H(\gamma) \circ S_H(b-\gamma) \leq S_H(b).\] 
If $a+\gamma \leq 0$, then again Lemma \ref{lem:bump} yields  \[S_H(a) \circ S_H(b) = S_H(a+ \gamma) \circ S_H(-\gamma) \circ S_H(b) \geq S_H(a+ \gamma) \circ S_H(b-\gamma).\]
 Finally, if $a \leq 0 \leq a+\gamma$ we have \[S_H(a) \circ S_H(b) \geq S_H(a+b) \geq S_H(a+\gamma) \circ S_H(b-\gamma).\]

The proof for $N>2$ follows  by induction on $N$. Let $\rho$ be piecewise linear on the same partition and coincide with $\zeta$ on $t_0, t_1$, and with $\xi$ on $t_2, \ldots, t_N$. The induction hypothesis then yields
\[
S_{H}^{\xi}(0,t_2)\leq S_{H}^{\rho}(0,t_2) \ \  \text{and} \ \ S_{H}^{\rho}(t_1,T)\leq S_{H}^{\zeta}(t_1,T)
\]
from which we deduce
\[
S_{H}^{\xi}(0,T)\leq S_{H}^{\rho}(0,T) \leq S_{H}^{\zeta}(0,T).
\]
 \end{proof}

\subsection*{A uniform modulus of continuity} To extend the information obtained about the possible cancellations and oscillations from piecewise linear  to arbitrary continuous paths, we need  a well-posed theory for the pathwise viscosity solutions. Such a theory has been developed by the last two authors in \cite{LS-book} and \cite{LS18+}. The former reference imposes  conditions on the joint  dependence of the Hamiltonians in $(p,x)$ but does not require convexity. A special (resp. a more general) class  of convex or concave Hamiltonians, which do not require such conditions, is studied in Friz, Gassiat, Lions and Souganidis \cite{FGLS16} (resp.  Lions and Souganidis \cite{LS-college}).  
An alternative, although less intrinsic,  approach is to show that the solution operator has a unique extension from piecewise linear paths to arbitrary continuous ones. %
\smallskip

To avoid stating additional conditions and since finding the optimal assumptions on the joint dependence on $(p,x)$ of the Hamiltonians is not the main focus of this paper, we bypass this issue here. Instead, we formulate a general assumption that allows to have a unique extension of the solution operator to all continuous paths, which is enough to analyze the domain of dependence. We only remark that this assumption is satisfied by the Hamiltonians considered in \cite{LS18+} as well as some other ones that can be analyzed by the same methods. 
\smallskip

For $t \in (0,T)$, the minimal action, also known as the fundamental solution, associated with  Hamiltonians satisfying \eqref{takis1} is given by 
$${\mathcal L}( x, y, t): = \inf\{ \int_0^t L(\dot{\gamma}(s),{\gamma}(s)) ds:\;\; \gamma \in C^{0,1}([0,T]) \ \text{such that} \ \gamma(0)=x, \gamma(t)=y \}.$$
When we need to emphasize the dependence of $\mathcal L$ \blue{on $H$},  we write ${\mathcal L}^H$.
\smallskip

We recall (see, for example, \cite{LionsBook}) that, for all $t, s \geq 0$ and  $x,y,z \in \R^d$, 
\begin{equation} \label{eq:dpp}
{\mathcal L}(x,z, t+s) \leq  {\mathcal L}(x,y, t) + {\mathcal L}(y,z, s).
\end{equation}
 
Moreover, for any $u_0 \in \text{BUC}(\R^d)$, $t \geq 0$ and $x \in \R^d,$
\begin{equation}\label{takis30}
u(x,t)=S_H(t)u_0(x)= \sup_{ y \in \R^d} \left[ u_0(y) - {\mathcal L}(x,y, t)\right].
\end{equation}

Finally, since  $-S_{-H}(t)u_0=S_{\widecheck{H}}(-u_0)$  with $\widecheck{H}(p,x)=H(-p, x)$, we also have, for any $u_0 \in \text{BUC}(\R^d)$, $t \geq 0$ and $x \in \R^d,$
\begin{equation}\label{takis31}
S_{-H}(t)u_0(x)=\inf_{y\in \R^d}  \left[ u_0(y) + {\mathcal L}^{\widecheck{H}}(x,y,t)\right] \blue{=\inf_{y\in \R^d}  \left[ u_0(y) + {\mathcal L}(y,x,t)\right]}.
\end{equation}



We assume that, for all $r>0$,  
\begin{equation} \label{eq:asnA0r}
\;\;\limsup_{\delta \to 0} \inf_{r \leq |x-y|} {\mathcal L}(x,y, \delta) = +\infty,
\end{equation}
and 
\begin{equation} \label{eq:asnAt0}
\lim_{\delta \to 0} \lim_{ r \to 0} \sup_{|x-y| \leq r}{\mathcal L}(x,y, \delta) = 0.
\end{equation}
Note that  \eqref{eq:asnAt0} is some sort of controllability assumption, while \eqref{eq:asnA0r} follows from a uniform in $x$ upper bound on $H$.

%
%

\begin{prop} \label{prop:modulus}
If \eqref{eq:asnA0r} and \eqref{eq:asnAt0} hold, then, for each $u_0$ $\in$ $BUC(\R^d)$ and $T \geq 0$, the family
\[ \left\{ S_H^\xi (0,T)u_0: \;\;\; \xi \mbox{ piecewise linear} \right\}
\]
has a uniform  modulus of continuity.
\end{prop}

%
%

The claim above 
is a consequence of the following estimate.

\begin{prop} \label{prop:Aestimate}
 Let $u=S^\xi_H(0,t)u_0$ with $\xi$ piecewise linear and $u_0$ $\in$ $BUC(\R^d)$. Then, for all $ t \geq 0$ and all  $x,y \in \R^d$,  \;\;
\begin{equation} \label{eq:estimate}
u(x,t) - u(y,t) \leq \inf_{\delta >0 } \left( {\mathcal L}(y,x, \delta) + \sup_{x',y' \in \R^d} \left[u_0(x') - u_0(y') -  {\mathcal L}(y',x',\delta) \right] \right). 
\end{equation}
\end{prop}
 
\begin{proof}
%
By induction it is enough to prove the estimate for $u=S_H(t)u_0$ and $u=S_H(-t)u_0$.
\smallskip

We begin  with the former and we fix $x, y \ \text{and} \ x_1 \in \R^d$. Assuming in what follows  the $\inf$ in the definition of $\mathcal{L}$ is attained, otherwise we work with approximate minimizers, we choose 
%
 $\gamma$ to be a  minimizer for ${\mathcal L}(y,x_1, t+\delta)$ and set $\tilde{y}_1= \gamma(t)$. It follows from  \eqref{eq:dpp} that 
\begin{align*}
{\mathcal L} (y,x, \delta) +{\mathcal L}(x,x_1,t)  \geq {\mathcal L}(y,x_1, t+\delta) & = \int_0^t L(\dot{\gamma}(s),\gamma(s)) ds +  \int_t^{t+\delta} L(\dot{\gamma}(s),\gamma(s)) ds\\
&\geq {\mathcal L}(y,\tilde{y}_1,t) + {\mathcal L}(\tilde{y}_1,x_1,\delta).
\end{align*}
Hence 
\begin{align*}
&\left( u_0(x_1) - {\mathcal L}(x,x_1, t)\right)  - \sup_{y_1 \in \R^d}\left\{  u_0(y_1) - {\mathcal L}(y,y_1,t) \right\} - {\mathcal L}(y,x, \delta) \\
\leq& \;  u_0({x}_1) -  u_0(\tilde{y}_1)  -  {\mathcal L}(x,x_1,t) + {\mathcal L}(y,\tilde{y}_1,t)   - {\mathcal L}(y,x, \delta) \\
\leq &\;  u_0({x}_1) -  u_0(\tilde{y}_1)  - {\mathcal L}(\tilde{y}_1,x_1,\delta)
\leq \; \sup_{x',y' \in \R^d} \left\{u_0(x') - u_0(y') -  {\mathcal L}(y',x',\delta) \right\}.
\end{align*}
It follows that 
\[u(x,t)-u(y,t) - {\mathcal L}(y,x, \delta) \leq \sup_{x',y'} \left\{u_0(x') - u_0(y') -  {\mathcal L}(y',x',\delta) \right\} \]
and we conclude by taking the inf over $\delta$.
\smallskip

In view of \eqref{takis31}, a similar argument gives the estimate for $u=S_{-H}(t)u_0$. 

\end{proof}

\begin{proof}[Proof of Proposition \ref{prop:modulus}]
Fix $u_0$ and let 
$$\eta(\delta): =  \sup_{x', y' \in \R^d} \left(u_0(x') - u_0(y') -    {\mathcal L}(y',x', \delta) \right)$$ 
and
$$\nu(x,y) := \inf_{\delta >0}  \left( {\mathcal L}(y,x, \delta) + \eta(\delta) \right),\;\; \; \omega(r) := \sup_{|x-y| \leq r} \max [\nu(x,y), \nu(y,x)].$$
If follows from Proposition \ref{prop:Aestimate} that, if $v=S_H^\xi u_0$, then $\left|v(x) - v(y) \right| \leq \omega(|x-y|)$. On the other hand, in view of \eqref{eq:asnA0r}, $\underset{{\delta \to 0}}\lim
\eta(\delta) = 0$, and, hence, using  \eqref{eq:asnAt0} we conclude that $\underset{{r \to 0}}\lim
\omega(r) = 0$.

\end{proof}

\subsection*{Extension and reduction} The extension result is stated next. In this subsection, we always assume that either $H$ is independent of $x$ or that \eqref{eq:asnA0r} and \eqref{eq:asnAt0} hold. In what follows we write $\|\cdot\|_{\infty, \mathcal O}$ for the $L^\infty$-norm over $\mathcal O$. 

\begin{cor}
\label{cor:continuity}The map $\xi\mapsto S_{H}(\xi)$ is uniformly continuous in the sup-norm in the sense that, if $(\xi^n)_{n\in \N}$ is a sequence of  piecewise-linear functions on $[0,T]$ with $\lim_{n,m\to\infty}\|\xi^n-\xi^m\|_{\infty,[0,T]}= 0$, then, for all $u \in BUC(\R^d)$, 
\begin{equation}\label{takis40}
\underset{n,m\to\infty}\lim 
 \|S^{\xi^n}_H(0,T)u-S^{\xi^m}_H(0,T)u\|_{\infty,\R^d} = 0.
\end{equation}
\end{cor}

It follows that  we can extend $\xi \mapsto S_H(\xi)$ to all continuous paths.  Indeed $\xi^n \to \xi$ in sup-norm as $n \to \infty$, then
\begin{equation}\label{takis41}
S^{\xi}_H(0,T)u:=\underset{n \to\infty}\lim S^{\xi^n}_H(0,T)u.
\end{equation}

{\it Proof of Corollary~\ref{cor:continuity}} 
\ Fix  $\d>0$ and let $\xi,\zeta$ be piecewise linear such that  $\|\xi-\zeta\|_{\infty}\le\d$ on $[0,T]$.  We extend $\xi,\zeta$ to all of $\R$ as constants on $(-\infty,0)$ and $(T,+\infty)$ 
and choose   $\eta \in[-1,1]$ 
such that $\xi(T)=\zeta(T)+\eta\delta$.
\smallskip

Let $\xi^{\pm\delta}$ be  defined by 
\[
\xi^{\pm\delta} := \begin{cases}  \xi \pm \delta  \  \mbox{ on }  \
[0,T],  \\[1.5mm]
 \xi  \  \mbox{ on } \ (-\infty, - \delta) \cup (T+\delta, + \infty), \end{cases}
 \]
and 
\[
\dot{\xi}^{\pm \d} = \pm 1 \  \mbox{ on }(-\delta,0) \ \ \text{and} \ \ \dot{\xi}^{\pm \d} = \mp 1 - \eta \mbox{ on }(T,T+\delta).
\]
It follows that  $\xi^\pm(-\delta)=\zeta(-\delta)$, $\xi^\pm(T+\delta)=\zeta(T+\delta)$, $\xi^{-\d}\le\zeta\le\xi^{\d}$, and, since $S_{H}^{\xi^{\pm\d}}(- \delta,T+\delta) = S_H( \mp \delta -\eta\d) \circ S_H^\xi(0,T) \circ S_H(\pm \delta)$, Corollary \ref{cor:Monotone} yields 
 \begin{equation*}\label{takis50}
 S_H(\delta(1-\eta)) \circ S_H^\xi(0,T) \circ S_H(-\delta) \leq S_{H}^{\zeta}(0,T) \leq S_H(-\delta(1+\eta)) \circ S_H^\xi(0,T) \circ S_H(\delta)
\end{equation*}
%
which implies that
\[
S_H^\xi(0,T) - \S_H(-\delta(1+\eta)) \circ S_H^\xi(0,T) \circ S_H(\delta)
\le S_{H}^{\xi}(0,T) - S_{H}^{\zeta}(0,T), 
\]
and 
\[
S_{H}^{\xi}(0,T) - S_{H}^{\zeta}(0,T) \le S_{H}^{\xi}(0,T) -  S_H(\delta(1-\eta)) \circ S_H^\xi(0,T)\circ S_H(-\delta).
\]
We now need to check that both sides of the above inequality go to $0$ as $\delta \to 0$. This  follows if
\[\underset{\delta \to 0}\lim \|S_H(\delta)\circ S_H^\xi(0,T)u - S_H(-\delta)\circ S_H^\xi(0,T)u\|_{\infty, \R^d} =0
\]
independently  of $\xi$, which is a consequence of  Proposition~\ref{prop:modulus}.

\qed
 

%
The next conclusion is an immediate consequence of Lemma \ref{lem:reduction3} and Corollary \ref{cor:continuity}. 
\begin{cor}
\label{cor:reduction1}Let $\xi$ be a continuous path such that $\xi_{T}=\max_{[0,T]}\xi$ and  $\xi_{0}=\min_{[0,T]}\xi$. Then, 
\[
S_{H}^{\xi}(0,T)=S_{H}(\xi_{0,T}).
\]
Similarly, if $\xi_{T}=\min_{[0,T]}\xi$ and  $\xi_{0}=\max_{[0,T]}\xi$, then
\[
S_{H}^{\xi}(0,T)=S_{-H}(-\xi_{0,T}).
\]
 
\end{cor}

It follows that we can have  a general representation for the solution to \eqref{eq:HJ-intro} as a (countable) composition of the flows $S_{H}(t),S_{H}(-t)$.

\begin{thm}
\label{thm:reduction}Let $\xi$ be a continuous path. Then 
\[
S^{\xi}(0,T)=S^{R_{0,T}(\xi)}(0,T).
\]
\end{thm}

\begin{proof}
We apply Corollary \ref{cor:reduction1} inductively to the successive extrema as defined in \eqref{eq:extrema1}, \eqref{eq:extrema2}, \eqref{eq:extrema3}. It only remains to show that this procedure converges for $i\to\pm\infty$. This follows from the continuity of $\xi$ in combination with Corollary \ref{cor:continuity}.

\end{proof}

%
%
%
%
%

\section{The optimality of the domain of dependence}

We consider the initial value problem 
\begin{equation}\label{takis60}
du=|Du|\cdot d{\xi} \ \text{in } \ \R^{d}\times (0,T] \quad u(\cdot,0)=u_0(\cdot) \ \text{in} \  \R^d
\end{equation} 
and prove Theorem \ref{thm:opt}. 

We remark that, in view of the geometric properties of \eqref{takis60}, it is enough to consider the evolution of the level set 
\[
P^{+}(t)=\left\{ x\in\R: \;\;u(x,t)\geq0\right\} .
\]
Indeed, \eqref{takis60} is a level-set PDE, that is, if $u$ is a solution, then also $\Phi(u)$ is a solution\blue{, for each smooth increasing reparametrization $\Phi \in C^1(\R)$}. At this point the choice of the Stratonovich differential in \eqref{takis60} is important (see Souganidis~\cite{S97,S16} and Lions, Souganidis \cite{LS-book} \blue{for more background on level-set PDE and SPDE}). It follows that $P^{+}(t)$ depends only on $P^{+}(0)$ and not on the particular form of $u_0$. In fact, in the case of \eqref{takis60} this can be read off the explicit solution formula, for all $\delta>0$, 
\[
S_{|\cdot|}(\delta)u(x)=\sup_{|x-y|~\leq\delta}u(y),\;\;\;\;\;S_{|\cdot|}(-\delta)u(x)=\inf_{|x-y|~\leq\delta}u(y).
\]
In particular, in  $d=1$ and with the convention that $[c,d]=\emptyset$ if $c>d$, it follows that, for all $\delta \in \R$,
\begin{equation}
\mathcal{S}_{|\cdot|}(\delta)([a,b])=[a-\delta,b+\delta].\label{eq:ev_intervals}
\end{equation}

We notice that, informally, for general initial conditions, $P^+$ expands with  speed $|d{\xi}|$ when $d {\xi}>0$, and contracts with speed $|d {\xi}|$ when $d{\xi}<0$.
\smallskip

The key behind the construction of the lower bound is the observation, already made in \cite{LS-book}, that there is some irreversibility in the dynamics. For example, once a hole is filled, that is two connected components of $P^{+}$ are joined by an increase in $\xi$, it cannot be recreated later when $\xi$ decreases. Symmetrically, if a component of $P^{+}$ is destroyed by a decrease in $\xi$,  it does not re-appear later.  This intuition leads to the lower bound for $\rho_H (\xi,T)$ derived below.
\smallskip

In what follows, to simplify the notation we omit the dependence of the solution operator and the speed of propagation on $H$, that is, we simply write $S$, $S^\xi$ and $\rho(\xi, T)$. We fix $d=1$ and establish first the lower bound in Theorem \ref{thm:opt}
, and then look at  the upper bound. 
Note that considering initial conditions depending only on the first coordinate implies that   the lower  bound also holds for  $d\geq 2$. 

\subsection*{Lower bound for the speed of propagation} The result is stated next.

\begin{prop}
Let $\xi$ be a continuous path. Then
\begin{equation}
\rho_{H}(\xi,T)\geq\|\tilde{R}^\xi(0,T)\|_{TV([0,T])}.
\end{equation}
\end{prop}

\begin{proof} 
Without loss of generality we assume that $\xi$ is a reduced path. Moreover, since the claim stays the same  if we replace $\xi$ by $-\xi$,  we further assume that $\xi(\tau_{0})=\max_{0\le s\le T}\xi(s)$. 
\smallskip

We first consider the case where $N:=\max\{n \leq 0:\ \tau_{n}=0\}$ is finite. Since $\xi$ is constant if $N=0$, we further assume $N \leq -1$ and fix a sequence $x_i$, $N - 1\leq i \leq 1$ such that  $x_1=0$ and, for all $N < k \leq  0,$
\begin{equation} \label{eq:xConstraint}
2 \left| \xi_{0,\tau_{k-2}} \right| <  x_{k+1} - x_k < 2  \left| \xi_{0,\tau_{k}} \right|.
\end{equation}
Set 
\begin{equation} \label{eq:defIk}
 I_k = \left\{ \begin{array}{ll} [x_{2k-1},x_{2k}] & \mbox{ if } 2k-1 >N \\ \emptyset & \mbox{ otherwise}, \end{array} \right. 
\end{equation}
and 
\begin{equation} \label{eq:defP}
P^{1}= \bigcup_{k \leq 0} I_k \cup [0,+\infty), \;\;\;\;\;P^{2}=(-\infty,x_N]\cup P^{1}.
\end{equation}

\begin{figure}
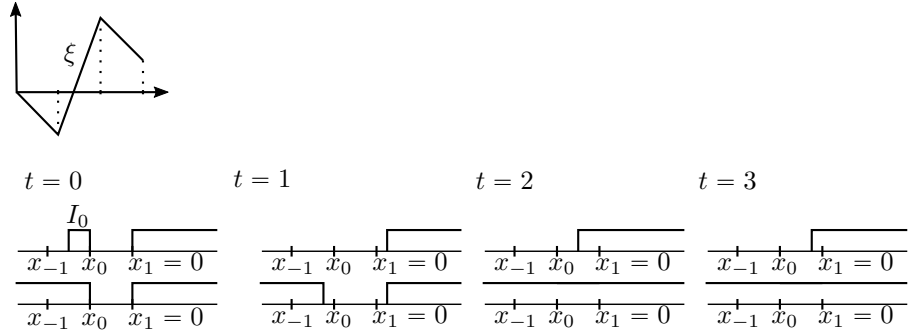
\caption{Lower bound}
\label{fig:lower_bound}
\end{figure}

Since $\xi$ is a reduced path and due to \eqref{eq:ev_intervals}, the evolution of $P^{1}$, $P^{2}$ can be easily described by induction on $k$ as follows.
\smallskip
 
The component $I_{k}$ evolves individually,  that is it does not intersect any other connected components,  before $\tau_{2k}$. This follows from the fact that $\left|x_{2k-1} - x_{2k-2}\right|$ and $\left|x_{2k} - x_{2k+1}\right|$ are smaller than $2 \xi_{0,\tau_{2k}}$.
\smallskip
 
Since $ - 2\xi_{0,\tau_{2k-3}} < \left|x_{2k} - x_{2k-1}\right| < - 2 \xi_{0,\tau_{2k-1}}$,  the component $I_{k}$ of $P^{1}$ disappears at time $\tau_{2k-1}$ but not at any of the earlier $\tau_{i}$'s.
\smallskip

Finally, given that $x_{2k-1} - x_{2k-2} < 2 \xi_{0,\tau_{2k-2}}$,  the component $I_{k}$ of $P^{2}$ has joined the components $I_{j}$ with $j<k$ by the time $\tau_{2k-2}$.
\smallskip

It follows  that 
$$\mathcal{S}(\xi,\tau_0)(P^{1})=[- \xi_{0,\tau_0}, +\infty) \ \ \text{and} \ \  \mathcal{S}^\xi(0,T)(P^{1})=[ - \xi_{0,T}, +\infty) $$
and 
 $$\mathcal{S}^\xi(0,T)(P^{2}) = \R.$$
 
Since $P^{1}$ and $P^{2}$ only differ for $x \leq x_N$, this implies $$\rho_{H}(\xi,T)\geq \left(- \xi_{0,T} - x_N\right)_+.$$

Choosing the $(x_k - x_{k-1})$ as large as possible in \eqref{eq:xConstraint} we obtain 
\begin{equation*}
\rho(\xi,T)\geq -\xi_{0,T} + 2 \sum_{k\leq 0} \left| \xi_{0,\tau_k} \right| = |\xi_{\tau_0,T}| + \sum_{k\leq 0} |\xi_{0,\tau_k} | +  \sum_{k\leq -1} \left| \xi_{0,\tau_k} \right|.
\end{equation*}
Finally, using that, for $k \leq 0$, $\left| \xi_{\tau_{k-1},\tau_k} \right| =  \left| \xi_{0,\tau_k} \right|  +  \left| \xi_{0,\tau_{k-1}} \right| $, we obtain
\begin{equation*}
\rho(\xi,T)\geq |\xi_{\tau_0,T}| + \sum_{k\leq 0} |\xi_{\tau_{k-1},\tau_k} |= \left\| \tilde{R}_{0,T}(\xi)\right\|_{TV}
\end{equation*}
which concludes the proof in the case where $\tau_N = 0$ for some $N$.
\smallskip

Next we treat the general case. We fix $N \leq -1$ arbitrary, and as before we define $x_k$, $I_k$ and $P^1$, $P^2$ satisfying \eqref{eq:xConstraint}, \eqref{eq:defIk} and \eqref{eq:defP}. Now since $\eqref{eq:xConstraint}$ implies that 
\[ x_{2k} - x_{2k-1} > - \inf_{[0,\tau_N]} \xi, \;\; x_{2k+1} - x_{2k} > \sup_{[0,\tau_N]} \xi,
\]
the $I_k$'s do not interact before time $\tau_N$.
\smallskip

Let $\tilde{x}_k = x_k +(-1)^k \xi_{0,\tau_N}$ and
\begin{equation*}
 I_k = \left\{ \begin{array}{ll} [\tilde{x}_{2k-1},\tilde{x}_{2k}] & \mbox{ if } 2k-1 >N \\ \emptyset & \mbox{ otherwise}.\end{array} \right. 
\end{equation*}
Then
\[ \mathcal{S}^\xi(0,\tau_N) (P^1) = \cup_{k \leq 0} \tilde{I_k} \cup [-\xi_{0,\tau_N},+\infty) , \;\;\;  \mathcal{S}^\xi(\tau_N,T) (P^2) =  (-\infty, \tilde{x}_N ]\cup_{k \leq 0} \tilde{I_k}.
\]
Note that
\begin{align*}
\tilde{x}_{k+1} - \tilde{x}_{k} = x_{k+1} - x_{k} + (-1)^k 2 \xi_{0,\tau_N} 
\end{align*}
is bounded from above by
$$ 2 \left| \xi_{0,\tau_{k}}\right|  + (-1)^k 2 \xi_{0,\tau_N} = 2  \left| \xi_{\tau_N,\tau_{k}}\right|$$
and, when $k \geq N+2$,  from below by $2  \left| \xi_{\tau_N,\tau_{k-2}}\right|$. 
\smallskip

Hence, the evolution on $[\tau_N,T]$ is then given as in the case $\tau_N=0$, and we obtain again
$$\mathcal{S}^\xi(0,T)(P^{1})=[ - \xi_{0,T}, +\infty) \ \ \text{and} \ \ 
\mathcal{S}^\xi(0,T)(P^{2})  = \R.$$
Taking again t$x_k - x_{k-1}$ as large as possible yields
\begin{equation*}
\rho(\xi,T)\geq \left\| \tilde{R}_{0,T}(\xi)\right\|_{TV([\tau_N,T])},
\end{equation*}
and letting $N \to - \infty$ finishes the proof.
\end{proof}

\subsection*{Optimality in one space dimension}
We assume $d=1$ and consider $x$-independent Hamiltonians. In this case the representation obtained in Section \ref{sec:Representation-with-reduced} is even simpler, since only the fully reduced path is needed.
\smallskip

\begin{prop}\label{cor:upper_bound_d1} Let $H:\R \to \R$ be continuous and  convex. Then 
\begin{equation} \label{eq:fully-reduced-soln}
u(\cdot,T)=S^{\xi}(0,T)\blue{u_0}=S^{\tilde{R}_{0,T}(\xi)}(0,T)\blue{u_0}, 
\end{equation}
and
\begin{equation}\label{eq:fully-reduced-bound}
\rho_{H}(\xi,T)\le \|H'\|_\infty \|\tilde{R}_{0,T}(\xi)\|_{TV([0,T])}.
\end{equation}
\end{prop}
\begin{proof}
  The claim is shown for $H$ smooth and strictly convex in 
  \cite{HKRS} using  a regularization argument. The result extends to convex $H$ by approximation, since \eqref{eq:HJ-intro} is stable under the passage to the limit in $H$ by standard viscosity theory; see \cite{LS-book}. 
\end{proof}

When  $d\ge2$, \eqref{eq:fully-reduced-soln} is not true in general. Indeed, this can be easily seen by the counter-example depicted in Figure~\ref{fig:no-full-reduction-1} and Figure~\ref{fig:no-full-reduction-2}, which correspond to the continuous, piece-wise linear path $\xi$ with
\[
 \dot\xi(t)~=\left\{ \begin{array}{ll}
4 & \mbox{ for }(0,1)\\
-2 & \mbox{ on }(1,2)\\
1 & \mbox{ on }(2,3)
\end{array}\right. \ \text{and} \ \dot{\tilde{R}}_{0,3}(\xi) =\left\{ \begin{array}{ll}
4 & \mbox{ on }(0,1)\\[1mm]
-\frac{1}{2} & \mbox{ on }(1,3);
\end{array}\right.
\]
it is easy to observe that in this case $S_{|\cdot|}^{\xi}(0,3) \ne S_{|\cdot|}^{\tilde{R}_{0,3}(\xi)}(0,3)$.
\smallskip

Note however that the claims in Proposition \ref{cor:upper_bound_d1} hold in arbitrary dimension for $H(p) = \frac{1}{2} |p|^2$ (cf.\ Gassiat and Gess~\cite{GG16})  and more generally for a class of uniformly convex $H$ (cf.\ Lions and Souganidis~\cite{LS18++}).

\begin{figure}
\includegraphics[width=4cm]{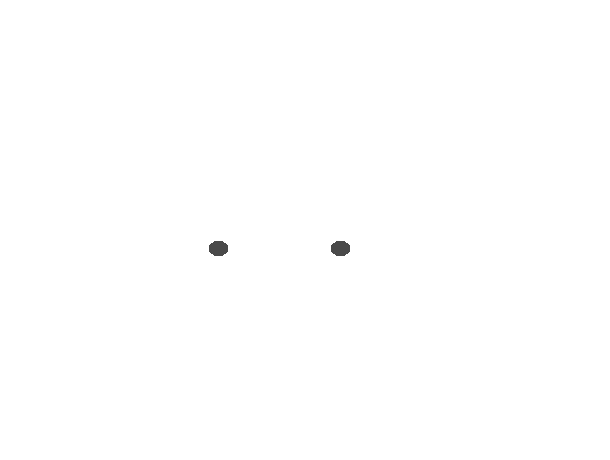}\includegraphics[width=4cm]{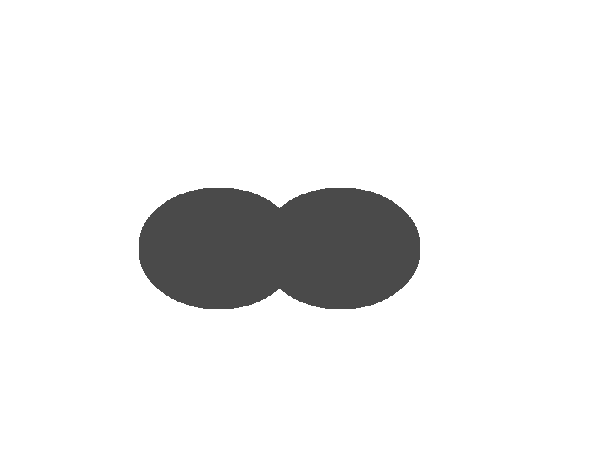}\includegraphics[width=4cm]{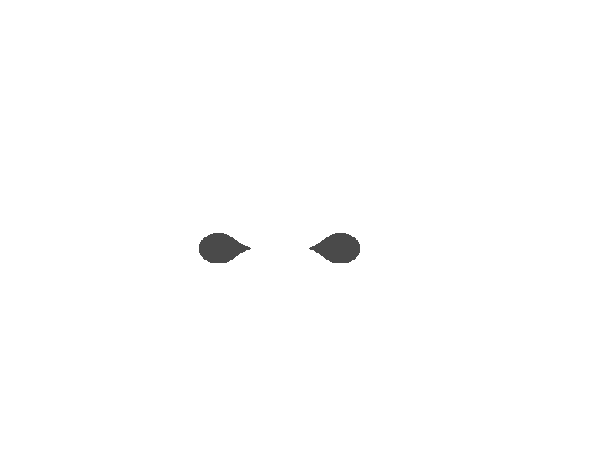}\includegraphics[width=4cm]{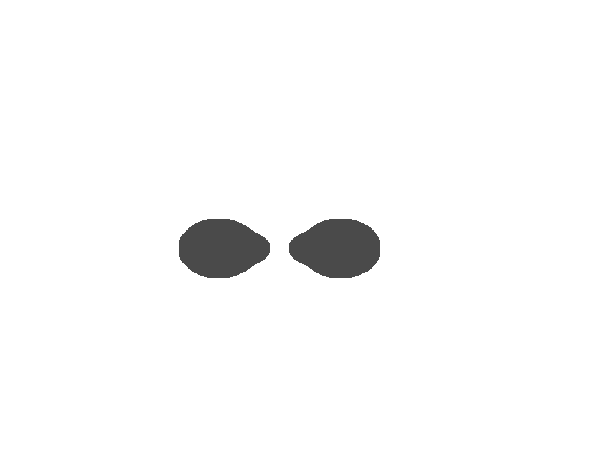}\caption{Evolution of $S_{|\cdot|}^{\xi}(0,t)$ at $t=0$, $t=1$, $t=2$, $t=3$}\label{fig:no-full-reduction-1}
\end{figure}

\begin{figure}
\includegraphics[width=4cm]{no-full-reduction-0}\includegraphics[width=4cm]{no-full-reduction-1}\includegraphics[width=4cm]{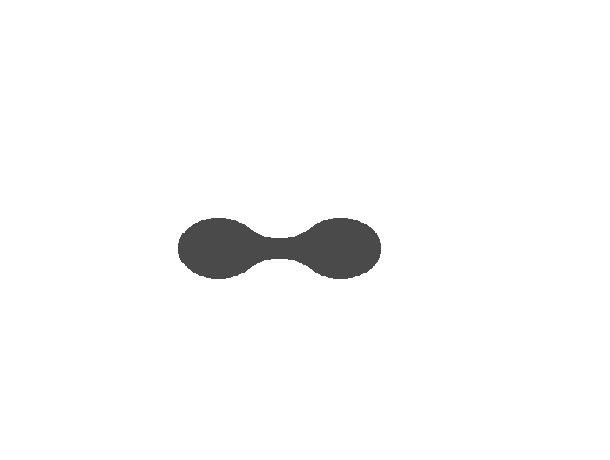} 
\caption{Evolution of $S_{|\cdot|}^{\tilde{R}(\xi)}(0,t)$ at $t=0$, $t=1$, $t=3$}\label{fig:no-full-reduction-2}
\end{figure}
\smallskip

Since \eqref{eq:fully-reduced-soln} is not valid in general in dimension $d \geq 2$, the speed of propagation may  depend on the total variation of the full reduced path $R_{0,T}(\xi)$. Note that we do not know, even for $H = |\cdot|$,  if one always has equality in \eqref{eq:main-speed} in that case.  The following proposition gives an example of a situation where $\rho_H(\xi,T) = \|R_{0,T}(\xi)\|_{TV([0,T])} >\|\tilde{R}_{0,T}(\xi)\|_{TV([0,T])}$. 

\begin{prop}
Let $\delta_1 > \delta_2 > \delta _3 > 0$ and $\xi$ continuous on $[0,3]$ with
\[
\dot{\xi} = \left\{ \begin{array}{ll} \delta_1 & \mbox{ on }(0,1), \\
- \delta_2 & \mbox{ on } (1,2), \\ 
+\delta_3 & \mbox{ on }(2,3). \end{array} \right.
\]
Then 
\begin{equation*}
\rho_{|\cdot|}(\xi,3)=\delta_1+\delta_2 + \delta_3=\|\xi\|_{TV([0,3])}.
\end{equation*}
\end{prop}

\begin{proof}
Fix $0 < L < 2 \delta_2$ and $\eta >0$, and consider initial conditions (see  Figure~\ref{fig:high-dim-propagation-1} and Figure~\ref{fig:high-dim-propagation-2}) 
\[ P_1 = [0,L] \times \{- \delta_1, + \delta_1\}, \;\;\;\; P_2 = P_1 \cup (\{-\delta_1\} \times \R). \] 

Then $(-\eta,0), (L+\eta,0) \not\in \mathcal{S}^{\xi}_{|\cdot|}(0,1)P_1$, and, hence,  $$(B_{\delta_2}(-\eta,0) \cup B_{\delta_2}(L+\eta,0)) \cap \mathcal{S}^{\xi}_{|\cdot|}(0,2)P_1 = \emptyset.$$
 Since $L<2\delta_2$, for $\eta$ small enough, the interior of $B_{\delta_2}(-\eta,0) \cap B_{\delta_2}(L+\eta,0)$ is non-empty. It follows that, for all $\eta_1 \in (0,\eta)$ small enough,
  $$B_{\delta_3}(L-\delta_2+\delta_3-\eta_1,0)\subseteq B_{\delta_2}(-\eta,0) \cup B_{\delta_2}(L+\eta,0),$$ and, consequently, $$(L-\delta_2+\delta_3-\eta_1,0) \not\in \mathcal{S}^{\xi}_{|\cdot|}(0,3)P_1.$$

We next note that $[-2\delta_1,L]\times [-2\delta_1,2\delta_1] \subseteq \mathcal{S}^{\xi}_{|\cdot|}(0,1)P_2$ and thus $$[-2\delta_1+\delta_2,L-\delta_2]\times [-2\delta_1+\delta_2,2\delta_1-\delta_2] \subseteq \mathcal{S}^{\xi}_{|\cdot|}(0,2)P_2.$$ It follows that $$B_{\delta_3}(L-\delta_2+\delta_3-\eta,0) \cap \mathcal{S}^{\xi}_{|\cdot|}(0,2)P_2 \ne \emptyset,$$ and, hence, for each $\eta>0$, $$(L-\delta_2+\delta_3-\eta,0) \in \mathcal{S}^{\xi}_{|\cdot|}(0,3)P_2.$$

In conclusion, for all $\eta>0$ small enough,
\[ (L- \delta_2 + \delta_3 - \eta, 0) \in \mathcal{S}^{\xi}_{|\cdot|}(0,3)(P^2) \setminus  \mathcal{S}^{\xi}_{|\cdot|}(0,3)(P^1) \]
so that $\rho_{|\cdot|}(\xi,3) \geq \delta_1 + L - \delta_2 + \delta_3 - \eta$. Letting $L \to 2 \delta_2$ and $\eta \to 0$ finishes the proof.

\begin{figure}
\includegraphics[width=4cm]{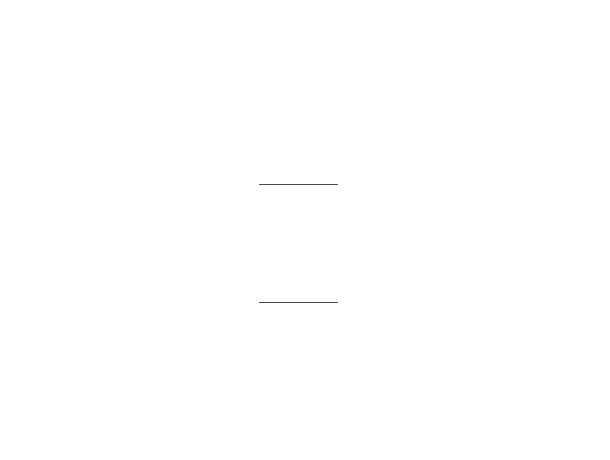}\includegraphics[width=4cm]{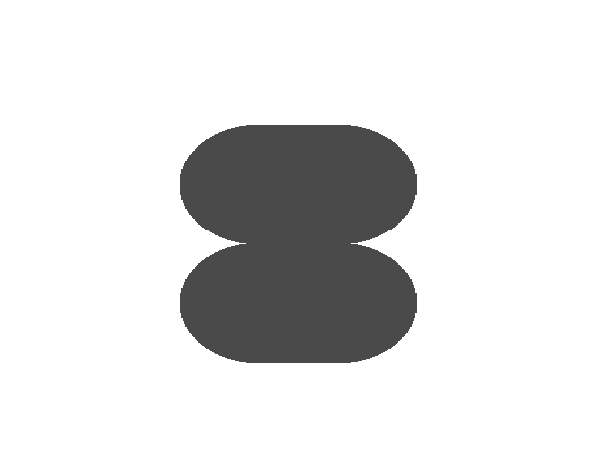}\includegraphics[width=4cm]{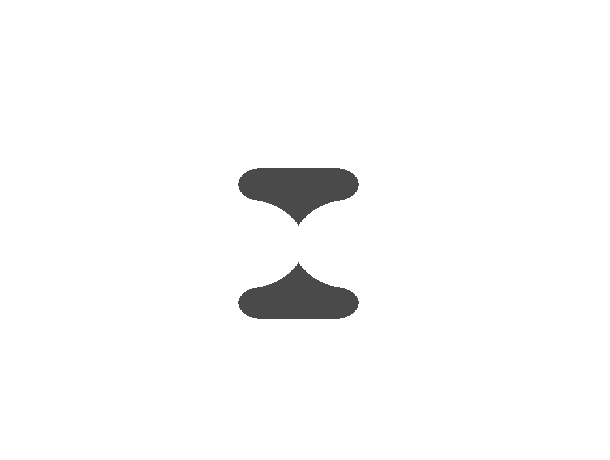}\includegraphics[width=4cm]{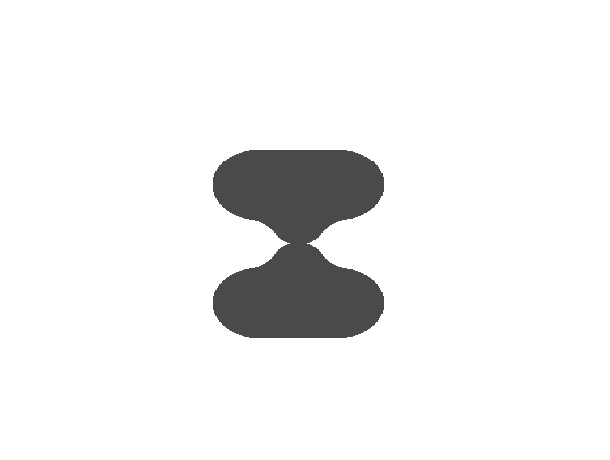}\caption{Evolution of $S_{|\cdot|}^{\xi}(0,\cdot)P_{1}$}\label{fig:high-dim-propagation-2}
\end{figure}

\begin{figure}
\includegraphics[width=4cm]{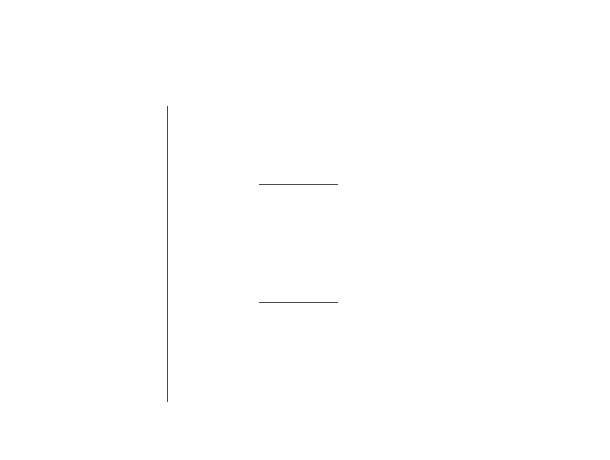}\includegraphics[width=4cm]{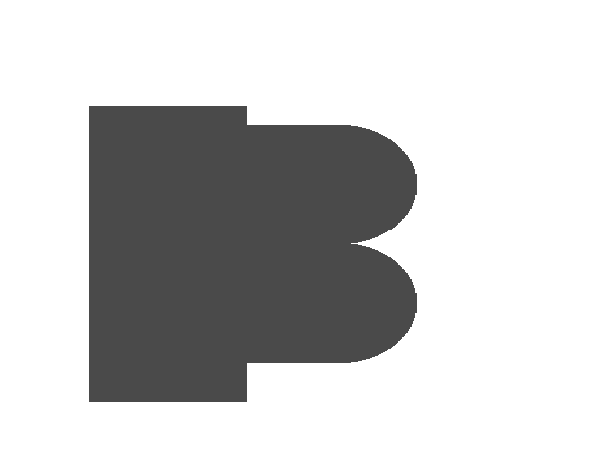}\includegraphics[width=4cm]{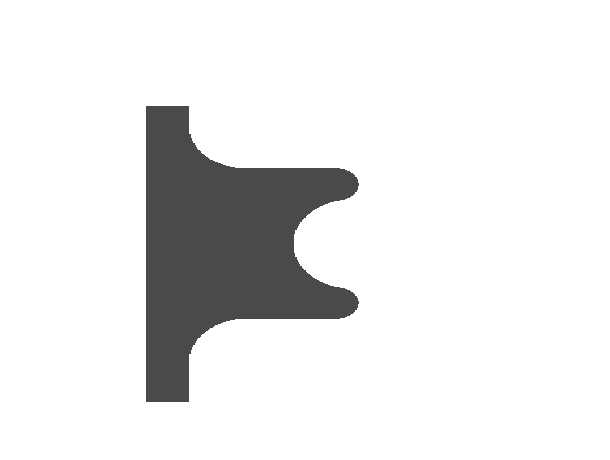}\includegraphics[width=4cm]{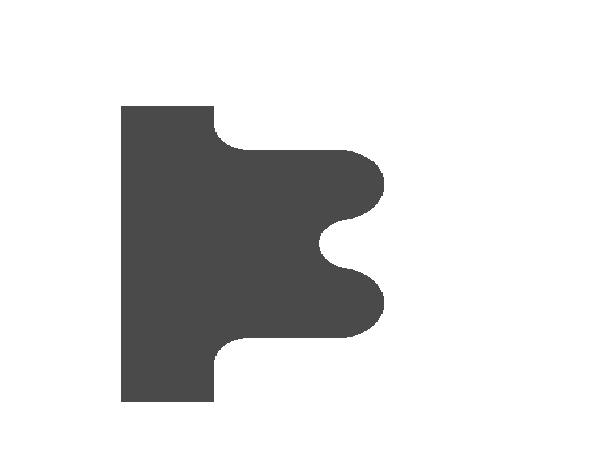}\caption{Evolution of $S_{|\cdot|}^{\xi}(0,\cdot)P_{2}$}\label{fig:high-dim-propagation-1}
\end{figure}

\end{proof}


\section{The Brownian case}

We begin with some preliminary discussion and a few results that are needed for the proof of Theorem~\ref{takis4}.
\smallskip

The key observation in the proof of Theorem \ref{takis4} is that the length of $R_{0,T}(B)$ on $[0,\tau_{0}]$, where  $\tau_0$ is given by \eqref{eq:extrema1}, is the same as that of a left-continuous path obtained by removing all excursions of $B$ between its minimum and maximum.
\smallskip

We fix an arbitrary continuous path $B:[0,+\infty)\to\R$ with $B(0)=0$. Let $M(t):=\sup_{s\leq t}B(s)$, $m(t):=\inf_{s\leq t}B(s)$, $R(t):=M(t)-m(t)$, recall that for $a\geq0$  
\[
\theta(a)=\inf\left\{ t\geq0,\;\;R(t)=a\right\}, 
\]
and,   $r\geq0$, define  
\[
S(r):=B(\theta(r)).
\]
Then $S$ is a left continuous process with right limits, the dynamics of which are simple to describe (see the proof of Lemma \ref{lem:RL} below): Away from  the jumps, $S$ has a drift given by $sign(S(r))$, and the jumps   are   given by $\Delta S(r):=S(r^{+})-S(r)=-sign(S(r))r$. In particular,  
\begin{equation}
L(r):=\left\Vert S\right\Vert _{TV([0,r])}=r+\sum_{0\leq s<r}s\mathbf{1}_{\Delta S(s)\neq0}.\label{eq:defL}
\end{equation}

The following lemma relates $L$ and  $R_{0,T}(B)$. 
\begin{lem}
\label{lem:RL} For all $T\geq0$, \[
\left\Vert R_{0,T}(B)\right\Vert _{TV([0,\tau_{0}])}=L(R(T)).
\]
\end{lem}

\begin{proof} Recall \eqref{eq:extrema1}, \eqref{eq:extrema2} and  \eqref{eq:extrema3} and note that
\[
\left\Vert R_{0,T}(B)\right\Vert _{TV([0,\tau_{0}])}=\sum_{i\leq0}\left\Vert R_{0,T}(B)\right\Vert _{TV([\tau_{i-1},\tau_{i}])}=\sum_{i\leq0}\left|B(\tau_{i-1})-B(\tau_{i})\right|.
\]

Fix $i\leq0$ and assume that $B(\tau_{i-1})=m(\tau_{i-1})<0$. If $r\in(R(\tau_{i-1}),R(\tau_{i})]$, then  monotonicity of $\theta$ and the definition of $\tau_i$ give $\theta(r)\leq\theta(R(\tau_{i}))=\tau_{i}$,  and, hence,  $m(\theta(r))=m(\tau_{i-1})$. 
\smallskip

Moreover, for $r\in(R(\tau_{i-1}),R(\tau_{i})]$, we have $B(\theta(r))=M(\theta(r))$ and thus 
\begin{align*}
  r=R(\theta(r))=M(\theta(r))-m(\tau_{i-1}) = B(\theta(r))-B(\tau_{i-1}).
\end{align*}
In conclusion, for all   $r\in(R(\tau_{i-1}),R(\tau_{i})],$
\[
S(r)=B(\theta(r))=B(\tau_{i-1})+r=S(R(\tau_{i-1}))+r,
\]
that is  $S$ has a jump of size $R(\tau_{i-1})$ at $r=R(\tau_{i-1})$ and is affine with slope $1$ on $(R(\tau_{i-1}),R(\tau_{i})]$.
\smallskip

If $B(\tau_{i-1})=M(\tau_{i-1})>0$, the same reasoning shows that $S$ has a jump of size $-R(\tau_{i-1})$ at $r=R(\tau_{i-1})$ and is  affine with slope $-1$ on $(R(\tau_{i-1}),R(\tau_{i})]$. 
\smallskip

Finally we get 
\begin{align*}
  L(R(\tau_0))
  &=\left\Vert S\right\Vert _{TV([0,R(\tau_0)])}
  =\sum_{i\leq0}\| S\|_{TV([R(\tau_{i-1}),R(\tau_i)])}\\
  &=\sum_{i\leq0}\left|R(\tau_{i-1})\right|+\left|R(\tau_{i})-R(\tau_{i-1})\right|
  =\sum_{i\leq0}R(\tau_{i})=\sum_{i\leq0}\left|B(\tau_{i})-B(\tau_{i-1})\right|.
\end{align*}
In view of the definition of $\tau_0$ and $R$ we have $R(\tau_0)=R(T)$ and thus $L(R(T))= L(R(\tau_0))$, which finishes the proof.
\end{proof}
Next we  assume that $B$ is a linear Brownian motion on a probability space $(\Omega,\mathcal{F},\mathbb{P})$ and we describe some of the properties of the time change $S$ that we will use below. The results have been  obtained in Imhof \cite[Theorem p. 352]{Imhof92} and Vallois \cite[Theorem 1]{Vallois95}.
\smallskip

The first result is that 
\begin{equation}
\left\{ s>0,\;\Delta S(s)\neq0\right\} \mbox{ is a Poisson point process on }(0,+\infty)\mbox{ with intensity measure }\frac{ds}{s}.\label{eq:PoissonS}
\end{equation}
For fixed $r>0$  set  $\sigma_{0}(r):=r$ and, for all  $k\in \N,$ define the successive jump times by
\[
\sigma_{k+1}(r)=\sup\{s<\sigma_{k}(r): \Delta S(s)\neq0\}, 
\]
It follows from  \eqref{eq:PoissonS} that  
\begin{equation}
\left(\sigma_{n+1}(r)/\sigma_{n}(r)\right)_{n\geq0}\overset{(d)}{=}(U_{n})_{n\geq0},
\label{eq:Unif}
\end{equation}
where $(U_{n})_{n\geq0}$ is a sequence of i.i.d.\ random variables each with  uniform distribution on $[0,1]$, and $\overset{(d)}{=}$ denotes equality in the sense of distributions.
\smallskip

Moreover, if 
\[
\theta_{n}(r):=\inf\left\{ t\geq0,R(t)=\sigma_{n}(r)\right\} ,
\]
then 
\begin{equation}\label{eq:Bessel}
\left( \left|B((\theta_{n+1}(r)+t)\wedge\theta_{n}(r))-B(\theta_{n+1}(r))\right| : t\geq0\right) _{n\geq0}
\overset{(d)}{=} \left( X^{n}(t\wedge\zeta_{n}(r));t\geq0\right) _{n\geq0}
\end{equation}
where the  processes $X^{n}$ are i.i.d. $3$-dimensional Bessel processes, independent from $S$, starting from $0$ and 
\[
\zeta_{n}(r):=\inf\left\{ t>0: \ X^{n}(t)=\sigma_{n}(r)\right\} .
\]

At this point we recall that,  in view of the scaling properties of the Brownian motion, it is enough to prove Theorem~\ref{takis4}  with  $T=1$, that is  to study the reduced path ${R}_{0,1}(B)$.
\smallskip

To prove \eqref{eq:RGaussian} we need the following two lemmata.

\begin{lem} \label{lem:LR1} There exists $C>0$ such that,  for all $r,x\geq0$,  
\begin{equation}
\P\left(R(1)\geq r| L(r)-r\geq x\right)\leq C\exp\left(-Cx^{2}\right).
\end{equation}
\end{lem}
\begin{proof}
Note that $R(1)\geq r$ is equivalent to $\theta(r)\leq1$.
\smallskip

It follows from \eqref{eq:Unif} and \eqref{eq:Bessel} that there exist i.i.d.\ uniformly distributed in $[0,1]$ random variables $U_{n}$ such that
\[
L(r)=r+\sum_{n>1}\sigma_{n}(r)=r\left(1+\sum_{n=0}^{\infty}U_{1}\cdots U_{n}\right).
\]
Moreover, in view of the scaling properties of the  Bessel processes, 
\[
\theta(r)=\sum_{n\geq0}~\left(\theta_{n}(r)-\theta_{n+1}(r)\right))\overset{(d)}{=}\sum_{n\geq0}\zeta_{n}(r))\overset{(d)}{=}\sum_{n\geq0}(\sigma_{n}(r))^{2}\tilde{\zeta}_{n},
\]
where 
\[
\tilde{\zeta}_{n}=\inf\left\{ t>0: \tilde{X}^{n}(t)=1\right\}.
\]
Since the i.i.d.\ Bessel random variables $\tilde{X}^{n}$ are independent from the random variables $U_{n}$, it follows that 
\[
(\theta(r),L(r))\overset{(d)}{=}\left(r^{2}\sum_{n=0}^{\infty}(U_{1}\ldots U_{n})^{2}\tilde{\zeta}_{n},r\left(1+\sum_{n=0}^{\infty}U_{1}\cdots U_{n}\right)\right).
\]
Using  once again scaling and the fact that Bessel processes have the Gaussian tails, we also find that, for some $C_0>0$ and all $n\in\N$, 
\begin{equation}
\P\left(\tilde{\zeta}_n^{-1/2}\geq y\right)=\P\left(\sup_{0\leq t\leq1}\tilde X^n_{t}\geq y\right)\leq\exp(-C_0y^{2}).\label{eq:tau12}
\end{equation}
Then, Jensen's inequality yields that, for any nonnegative sequence $\alpha_{n}$, 
\[
\left(\sum_{n\geq0}\alpha_{n}^{2}\tilde{\zeta}_{n}\right)\geq\left(\sum_{n\geq0}\alpha_{n}\right)^{3}\left(\sum_{n\geq0}\alpha_{n}^{1/2}\tilde{\zeta}_{n}^{-1/2}\right)^{-2}. 
\]
and thus
\[
\P\left(\sum_{n=0}^{\infty}\alpha_{n}^{2}\tilde{\zeta}_{n}\leq s\right)\leq\P\left(\sum_{n\geq0}\alpha_{n}^{1/2}\tilde{\zeta}_{n}^{-1/2}\geq s^{-1/2}\left(\sum_{n\geq0}\alpha_{n}\right)^{3/2}\right).
\]
On the other hand, \eqref{eq:tau12}, the independence of the $\tilde{\zeta}_{n}$ and a straightforward  Chernoff bound yield, for all $y \geq 0$ and some $C>0$, 
\[
\P\left(\sum_{n\geq0}\alpha_{n}^{1/2}\tilde{\zeta}_{n}^{-1/2}\geq y \right) \leq \exp\left(-C\frac{y^2}{\sum_{n\geq 0} \alpha_n}\right)
\] 
 for some $C>0$,
so that
\[
\P\left(\sum_{n=0}^{\infty}\alpha_{n}^{2}\tilde{\zeta}_{n}\leq s\right)\leq\exp\left(-Cs^{-1}\left(\sum_{n \geq 0}\alpha_{n}\right)^{2}\right).
\]
Applying the inequality above with  $\alpha_{n}=U_{1}\ldots U_{n}$ and $s=r^{-2}$,  we get 
\[
\P\left(\theta(r)\leq1| L(r)-r \ge x\right)\leq\exp\left(-Cr^{2}(x/r)^{2}\right)\leq\exp(-Cx^{2}).
\]
\end{proof}
\begin{lem} \label{lem:LR2} There exists $C>0$ such that,  for all $x\geq2r\geq0$ and  $\delta\in(0,1]$, 
\begin{equation}
\P\left(L(r(1+\delta))-L(r)-r\delta\geq x\right)\leq C\exp\left(-\frac{x}{2r}|\ln(\delta)|\right).
\end{equation}
\end{lem}
\begin{proof}
Note that 
\begin{align*}
L(r(1+\delta))-L(r)-r\delta= & \sum_{r\leq s < r(1+\delta)}s1_{\Delta S(s)\neq0}\\
\leq & \;\;\;2r \#\left\{ s\in [r,r(1+\delta)),\;\;\Delta S(s)\neq0\right\} .
\end{align*}
In view of  \eqref{eq:PoissonS}, $\#\left\{ s\in [r,r(1+\delta)),\;\;\Delta S(s)\neq0\right\} $ is a Poisson random variable with parameter
  $$\int_{r}^{r(1+\delta)}\frac{du}{u}=\ln(1+\delta)\leq\delta.$$ 

In addition,  any   Poisson random variable $P$ with parameter $\delta$ satisfies the classical inequality 
\begin{align*}
\P(P\geq y)&\leq\inf_{\gamma}\left\{ \E[e^{\gamma P}]e^{-\gamma y}\right\} =\inf_{\gamma}\exp(-\gamma y+\delta(e^{\gamma}-1))=\exp(-y\ln\left(\frac{y}{\delta}\right)+y-\delta)\\
&\leq C\exp(-y|\ln(\delta)|).
\end{align*}
It follows that 
\begin{align*}
\P\left(L(r(1+\delta))-L(r)-r\delta\geq x\right) & \leq C\exp\left(-\frac{x}{2r}|\ln(\delta)|\right).
\end{align*}
\end{proof}
%

We are now ready for the proof of Theorem~\ref{takis4}.

\begin{proof}[Proof of Theorem \ref{thm:BM}] We begin with \eqref{eq:RPoisson} and 
note that on the interval $[0,\theta(1)]$ we have $\tau_{0}=\theta(1)$. 
\smallskip

It then follows from the definition of $\theta(1)$ that   $R(\tau_0)=R(\theta(1))=1$ and thus, using  Lemma \ref{lem:RL}, we find 
\[
\left\Vert R_{0,\theta(1)}(B)\right\Vert _{TV([0,\theta(1)])}=L(R(\theta(1)))=L(1).
\]
We further note that, in view of  \eqref{eq:Unif},  we have that
\[
L(1)\overset{(d)}{=}1+U_{0}+U_{0}U_{1}+U_{0}U_{1}U_{2}+\ldots,
\]
where the random variables $U_{i}$ are i.i.d.\ uniformly distributed on $[0,1]$. 
\smallskip

It now follows from  \cite[Theorem 3.1]{GG96} that $L(1)$ has Poissonian tails, that is, 
\[
\underset{x\to \infty}\lim \frac{\ln\P(L(1)\ge x)}{x\ln x}\to -1.
\]

We present now the proof of \eqref{eq:RGaussian}. Throughout the argument below, $C$ will denote a constant whose value may change from line to line.
\smallskip

We first note that 
\[
\left\Vert R_{0,1}(B)\right\Vert _{TV([0,1])}\leq\left\Vert R_{0,1}(B)\right\Vert _{TV([0,\tau_{0}])}+\left\Vert R_{0,1}(B)\right\Vert _{TV([\tau_{-1},1])}.
\]
Moreover, the symmetry of Brownian motion under time reversal gives  
\[
\left\Vert R_{0,1}(B)\right\Vert _{TV([\tau_{-1},1])}\overset{(d)}{=}\left\Vert R_{0,1}(B)\right\Vert _{TV([0,\tau_{0}])}. 
\]
It follows that it suffices to bound the tail probabilities of $\left\Vert R_{0,1}(B)\right\Vert _{TV([0,\tau_{0}]))}$ and, hence, in view of  Lemma \ref{lem:RL}, the tail of $L(R(1))$.
\smallskip

Fix $\gamma \in (0,2)$ and let $\alpha <\gamma$ be such that $\gamma = \frac{2(1+\alpha)}{3}$. For a fixed $x>0$, let $r_{0}:=\frac{x}{4}$ and, for $k \in \N$,   $r_{k+1}:=r_{k}(1+e^{-x^\alpha})^{-1}$. It is immediate that there exists an $N$ such that  $N\leq Ce^{Cx^\alpha}$ and  $r_{N}\leq x^{-1}$.
\smallskip

Moreover, 
\begin{align*}
\P\left(L(R(1))\geq x\right) & \leq\sum_{k=0}^{N-1}\P\left(r_{k+1}\leq R(1)\leq r_{k};L(r_{k})\geq x\right)+\P\left(R(1)\geq\frac{x}{4}\right)+\P\left(L(r_{N})\geq x\right).
\end{align*}
The second term on the right hand side is bounded by $\exp(-Cx^{2})$ since $R(1)$ has Gaussian tails. Moreover, Brownian scaling implies that $L(\tau t) \overset{(d)}{=} \tau L(t)$ for all $\tau>0$ and thus
\[
\P(L(r_{N})\geq x)\leq \P\left(L\big(\frac{1}{x}\big)\geq x\right) \leq \P(L(1)\geq x^{2}) \leq \exp(-Cx^{2}).
\]
In addition, we note that, for $k\in\{0,\ldots,N-1\}$, 
\begin{align*}
 & \;\;\;\;\;\P\left(r_{k+1}\leq R(1)\leq r_{k};L(r_{k})\geq x\right)\\
 & \leq\P\left(r_{k+1}\leq R(1);L(r_{k+1})\geq\frac{x}{2}\right)+\P\left(r_{k+1}\leq R(1);L(r_{k})-L(r_{k+1})\geq\frac{x}{2}\right).
\end{align*}
Lemma \ref{lem:LR1} implies that 
\[
\P\left(r_{k+1}\leq R(1);L(r_{k+1})\geq\frac{x}{2}\right)\leq\P\left(r_{k+1}\leq R(1); L(r_{k+1})-r_{k+1}\geq\frac{x}{4}\right)\leq C \exp(-Cx^{2}).
\]
Then, the Cauchy-Schwarz inequality and Lemma \ref{lem:LR2} give 
\begin{align*}
P\left(r_{k+1}\leq R(1);L(r_{k})-L(r_{k+1})\geq\frac{x}{2}\right) & \leq P\left(r_{k+1}\leq R(1)\right)^{1/2}\P\left(L(r_{k})-L(r_{k+1})\geq\frac{x}{2}\right)^{1/2}\\
 & \leq C\exp(-Cr_{k+1}^{2})\exp\left(-C\frac{x^{1+\alpha}}{r_{k+1}}\right)\\
 & \leq C\exp\left(-Cx^{\frac{2(1+\alpha)}{3}}\right).
\end{align*}
It follows that  
\[
\P\left(L(R(1))\geq x\right)\leq Ce^{Cx^\alpha}e^{-Cx^{\gamma}}\leq C e^{-Cx^{\gamma}}.
\]
\end{proof}

\section*{Acknowledgements}
Gassiat was partially supported by the ANR via the project
ANR-16-CE40-0020-01. Souganidis was partially supported by the National Science Foundation grant DMS-1600129 and  the Office for Naval Research grant N000141712095. Gess was partially supported by the DFG through CRC 1283.

 \bibliographystyle{plain}
\bibliography{Gassiat_Gess_Lions_Souganidis-2019-Speed_of_propagation}

\end{document}

%% file: reduced_path.pdf_tex
\begingroup%
  \makeatletter%
  \providecommand\color[2][]{%
    \errmessage{(Inkscape) Color is used for the text in Inkscape, but the package 'color.sty' is not loaded}%
    \renewcommand\color[2][]{}%
  }%
  \providecommand\transparent[1]{%
    \errmessage{(Inkscape) Transparency is used (non-zero) for the text in Inkscape, but the package 'transparent.sty' is not loaded}%
    \renewcommand\transparent[1]{}%
  }%
  \providecommand\rotatebox[2]{#2}%
  \ifx\svgwidth\undefined%
    \setlength{\unitlength}{268.50715697bp}%
    \ifx\svgscale\undefined%
      \relax%
    \else%
      \setlength{\unitlength}{\unitlength * \real{\svgscale}}%
    \fi%
  \else%
    \setlength{\unitlength}{\svgwidth}%
  \fi%
  \global\let\svgwidth\undefined%
  \global\let\svgscale\undefined%
  \makeatother%
  \begin{picture}(1,0.59543618)%
    \put(0,0){\includegraphics[width=\unitlength,page=1]{reduced_path.pdf}}%
    \put(0.64355598,-0.08044476){\color[rgb]{0,0,0}\makebox(0,0)[lb]{\smash{}}}%
    \put(0.64781237,-0.03575321){\color[rgb]{0,0,0}\makebox(0,0)[lb]{\smash{}}}%
    \put(0.62919085,0.18078793){\color[rgb]{0,0,0}\makebox(0,0)[lb]{\smash{$\tau_0$}}}%
    \put(0.84078216,0.22228464){\color[rgb]{0,0,0}\makebox(0,0)[lb]{\smash{$\tau_1$}}}%
    \put(0.93937152,0.15258971){\color[rgb]{0,0,0}\makebox(0,0)[lb]{\smash{T}}}%
    \put(0.41780853,0.22494488){\color[rgb]{0,0,0}\makebox(0,0)[lb]{\smash{$\tau_{-1}$}}}%
    \put(0.81593767,0.31592669){\color[rgb]{0,0,0}\makebox(0,0)[lb]{\smash{$\xi$}}}%
    \put(0.27325463,0.35263757){\color[rgb]{0,0,0}\makebox(0,0)[lb]{\smash{$R_{0,T}(\xi)$}}}%
    \put(0.05375549,0.22132692){\color[rgb]{0,0,0}\makebox(0,0)[lb]{\smash{$\tau_{-3}$}}}%
    \put(0.18463787,0.18142377){\color[rgb]{0,0,0}\makebox(0,0)[lb]{\smash{$\tau_{-2}$}}}%
    \put(0,0){\includegraphics[width=\unitlength,page=2]{reduced_path.pdf}}%
    \put(0.74211625,0.47206122){\color[rgb]{0,0,0}\makebox(0,0)[lb]{\smash{$\tilde R_{0,T}(\xi)$}}}%
  \end{picture}%
\endgroup%

%% file: reduction1.pdf_tex
\begingroup%
  \makeatletter%
  \providecommand\color[2][]{%
    \errmessage{(Inkscape) Color is used for the text in Inkscape, but the package 'color.sty' is not loaded}%
    \renewcommand\color[2][]{}%
  }%
  \providecommand\transparent[1]{%
    \errmessage{(Inkscape) Transparency is used (non-zero) for the text in Inkscape, but the package 'transparent.sty' is not loaded}%
    \renewcommand\transparent[1]{}%
  }%
  \providecommand\rotatebox[2]{#2}%
  \ifx\svgwidth\undefined%
    \setlength{\unitlength}{112.59844442bp}%
    \ifx\svgscale\undefined%
      \relax%
    \else%
      \setlength{\unitlength}{\unitlength * \real{\svgscale}}%
    \fi%
  \else%
    \setlength{\unitlength}{\svgwidth}%
  \fi%
  \global\let\svgwidth\undefined%
  \global\let\svgscale\undefined%
  \makeatother%
  \begin{picture}(1,0.69823688)%
    \put(1.59418396,-1.63590733){\color[rgb]{0,0,0}\makebox(0,0)[lb]{\smash{}}}%
    \put(1.60433394,-1.52933392){\color[rgb]{0,0,0}\makebox(0,0)[lb]{\smash{}}}%
    \put(0,0){\includegraphics[width=\unitlength,page=1]{reduction1.pdf}}%
    \put(0.5522523,0.44378217){\color[rgb]{0,0,0}\makebox(0,0)[lb]{\smash{$\xi$}}}%
    \put(0,0){\includegraphics[width=\unitlength,page=2]{reduction1.pdf}}%
    \put(0.34103361,0.16978463){\color[rgb]{0,0,0}\makebox(0,0)[lb]{\smash{$t_1$}}}%
    \put(0.52853347,0.15632771){\color[rgb]{0,0,0}\makebox(0,0)[lb]{\smash{$T$}}}%
    \put(0,0){\includegraphics[width=\unitlength,page=3]{reduction1.pdf}}%
    \put(0.00730188,0.54657853){\color[rgb]{0,0,0}\makebox(0,0)[lb]{\smash{$a_0 t_1$}}}%
    \put(0,0){\includegraphics[width=\unitlength,page=4]{reduction1.pdf}}%
    \put(-0.00884643,0.27385153){\color[rgb]{0,0,0}\makebox(0,0)[lb]{\smash{$\xi_{0,T}$}}}%
    \put(0,0){\includegraphics[width=\unitlength,page=5]{reduction1.pdf}}%
  \end{picture}%
\endgroup%

%% file: reduction2.pdf_tex
\begingroup%
  \makeatletter%
  \providecommand\color[2][]{%
    \errmessage{(Inkscape) Color is used for the text in Inkscape, but the package 'color.sty' is not loaded}%
    \renewcommand\color[2][]{}%
  }%
  \providecommand\transparent[1]{%
    \errmessage{(Inkscape) Transparency is used (non-zero) for the text in Inkscape, but the package 'transparent.sty' is not loaded}%
    \renewcommand\transparent[1]{}%
  }%
  \providecommand\rotatebox[2]{#2}%
  \ifx\svgwidth\undefined%
    \setlength{\unitlength}{118.15428246bp}%
    \ifx\svgscale\undefined%
      \relax%
    \else%
      \setlength{\unitlength}{\unitlength * \real{\svgscale}}%
    \fi%
  \else%
    \setlength{\unitlength}{\svgwidth}%
  \fi%
  \global\let\svgwidth\undefined%
  \global\let\svgscale\undefined%
  \makeatother%
  \begin{picture}(1,0.66554549)%
    \put(1.59103768,-1.55879166){\color[rgb]{0,0,0}\makebox(0,0)[lb]{\smash{}}}%
    \put(1.60071038,-1.45722953){\color[rgb]{0,0,0}\makebox(0,0)[lb]{\smash{}}}%
    \put(0,0){\includegraphics[width=\unitlength,page=1]{reduction2.pdf}}%
    \put(0.61519848,0.35813116){\color[rgb]{0,0,0}\makebox(0,0)[lb]{\smash{$\xi$}}}%
    \put(0,0){\includegraphics[width=\unitlength,page=2]{reduction2.pdf}}%
    \put(0.39681283,0.16199325){\color[rgb]{0,0,0}\makebox(0,0)[lb]{\smash{$t_1$}}}%
    \put(0.73109584,0.15258887){\color[rgb]{0,0,0}\makebox(0,0)[lb]{\smash{$T$}}}%
    \put(0,0){\includegraphics[width=\unitlength,page=3]{reduction2.pdf}}%
    \put(0.10271222,0.51850475){\color[rgb]{0,0,0}\makebox(0,0)[lb]{\smash{$a_0 t_1$}}}%
    \put(0,0){\includegraphics[width=\unitlength,page=4]{reduction2.pdf}}%
    \put(0.01454193,0.11818368){\color[rgb]{0,0,0}\makebox(0,0)[lb]{\smash{$\xi_{0,T}$}}}%
    \put(0,0){\includegraphics[width=\unitlength,page=5]{reduction2.pdf}}%
  \end{picture}%
\endgroup%

%% file: lower_bound.pdf_tex
\begingroup%
  \makeatletter%
  \providecommand\color[2][]{%
    \errmessage{(Inkscape) Color is used for the text in Inkscape, but the package 'color.sty' is not loaded}%
    \renewcommand\color[2][]{}%
  }%
  \providecommand\transparent[1]{%
    \errmessage{(Inkscape) Transparency is used (non-zero) for the text in Inkscape, but the package 'transparent.sty' is not loaded}%
    \renewcommand\transparent[1]{}%
  }%
  \providecommand\rotatebox[2]{#2}%
  \ifx\svgwidth\undefined%
    \setlength{\unitlength}{338.0350212bp}%
    \ifx\svgscale\undefined%
      \relax%
    \else%
      \setlength{\unitlength}{\unitlength * \real{\svgscale}}%
    \fi%
  \else%
    \setlength{\unitlength}{\svgwidth}%
  \fi%
  \global\let\svgwidth\undefined%
  \global\let\svgscale\undefined%
  \makeatother%
  \begin{picture}(1,0.36383013)%
    \put(0.31746135,-0.34330918){\color[rgb]{0,0,0}\makebox(0,0)[lb]{\smash{}}}%
    \put(0.32084227,-0.3078099){\color[rgb]{0,0,0}\makebox(0,0)[lb]{\smash{}}}%
    \put(0,0){\includegraphics[width=\unitlength,page=1]{lower_bound.pdf}}%
    \put(0.05987222,0.29778227){\color[rgb]{0,0,0}\makebox(0,0)[lb]{\smash{$\xi$}}}%
    \put(0.06223884,0.11776326){\color[rgb]{0,0,0}\makebox(0,0)[lb]{\smash{$I_0$}}}%
    \put(0.01576074,0.15496219){\color[rgb]{0,0,0}\makebox(0,0)[lb]{\smash{$t=0$}}}%
    \put(0.24884873,0.15747297){\color[rgb]{0,0,0}\makebox(0,0)[lb]{\smash{$t=1$}}}%
    \put(0.5189917,0.15508233){\color[rgb]{0,0,0}\makebox(0,0)[lb]{\smash{$t=2$}}}%
    \put(0,0){\includegraphics[width=\unitlength,page=2]{lower_bound.pdf}}%
    \put(0.13161526,0.06502048){\color[rgb]{0,0,0}\makebox(0,0)[lb]{\smash{$x_1=0$}}}%
    \put(0.07954965,0.06502048){\color[rgb]{0,0,0}\makebox(0,0)[lb]{\smash{$x_0$}}}%
    \put(0.018992,0.06502047){\color[rgb]{0,0,0}\makebox(0,0)[lb]{\smash{$x_{-1}$}}}%
    \put(0,0){\includegraphics[width=\unitlength,page=3]{lower_bound.pdf}}%
    \put(0.13160281,0.00348839){\color[rgb]{0,0,0}\makebox(0,0)[lb]{\smash{$x_1=0$}}}%
    \put(0.0795372,0.00348839){\color[rgb]{0,0,0}\makebox(0,0)[lb]{\smash{$x_0$}}}%
    \put(0.01899201,0.00348839){\color[rgb]{0,0,0}\makebox(0,0)[lb]{\smash{$x_{-1}$}}}%
    \put(0.40362851,0.00600293){\color[rgb]{0,0,0}\makebox(0,0)[lb]{\smash{$x_1=0$}}}%
    \put(0.35156289,0.00600293){\color[rgb]{0,0,0}\makebox(0,0)[lb]{\smash{$x_0$}}}%
    \put(0.29115316,0.005855){\color[rgb]{0,0,0}\makebox(0,0)[lb]{\smash{$x_{-1}$}}}%
    \put(0.65216983,0.0058427){\color[rgb]{0,0,0}\makebox(0,0)[lb]{\smash{$x_1=0$}}}%
    \put(0.60010422,0.0058427){\color[rgb]{0,0,0}\makebox(0,0)[lb]{\smash{$x_0$}}}%
    \put(0.53964812,0.005855){\color[rgb]{0,0,0}\makebox(0,0)[lb]{\smash{$x_{-1}$}}}%
    \put(0.90072938,0.00581803){\color[rgb]{0,0,0}\makebox(0,0)[lb]{\smash{$x_1=0$}}}%
    \put(0.84866377,0.00581803){\color[rgb]{0,0,0}\makebox(0,0)[lb]{\smash{$x_0$}}}%
    \put(0.78814308,0.005855){\color[rgb]{0,0,0}\makebox(0,0)[lb]{\smash{$x_{-1}$}}}%
    \put(0,0){\includegraphics[width=\unitlength,page=4]{lower_bound.pdf}}%
    \put(0.40377641,0.06502048){\color[rgb]{0,0,0}\makebox(0,0)[lb]{\smash{$x_1=0$}}}%
    \put(0.3517108,0.06502048){\color[rgb]{0,0,0}\makebox(0,0)[lb]{\smash{$x_0$}}}%
    \put(0.29115315,0.06502047){\color[rgb]{0,0,0}\makebox(0,0)[lb]{\smash{$x_{-1}$}}}%
    \put(0,0){\includegraphics[width=\unitlength,page=5]{lower_bound.pdf}}%
    \put(0.65227137,0.06502048){\color[rgb]{0,0,0}\makebox(0,0)[lb]{\smash{$x_1=0$}}}%
    \put(0.60020576,0.06502048){\color[rgb]{0,0,0}\makebox(0,0)[lb]{\smash{$x_0$}}}%
    \put(0.53964812,0.06502047){\color[rgb]{0,0,0}\makebox(0,0)[lb]{\smash{$x_{-1}$}}}%
    \put(0,0){\includegraphics[width=\unitlength,page=6]{lower_bound.pdf}}%
    \put(0.76748666,0.15508233){\color[rgb]{0,0,0}\makebox(0,0)[lb]{\smash{$t=3$}}}%
    \put(0,0){\includegraphics[width=\unitlength,page=7]{lower_bound.pdf}}%
    \put(0.90072408,0.06495709){\color[rgb]{0,0,0}\makebox(0,0)[lb]{\smash{$x_1=0$}}}%
    \put(0.84865847,0.06495709){\color[rgb]{0,0,0}\makebox(0,0)[lb]{\smash{$x_0$}}}%
    \put(0.78814307,0.06502047){\color[rgb]{0,0,0}\makebox(0,0)[lb]{\smash{$x_{-1}$}}}%
    \put(0,0){\includegraphics[width=\unitlength,page=8]{lower_bound.pdf}}%
  \end{picture}%
\endgroup%